\documentclass[brochure,english,12pt]{smfbourbaki}

\usepackage[T1]{fontenc}
\usepackage{lmodern,amssymb,bm}

\usepackage{babel}

\usepackage[utf8]{inputenc}

\usepackage[colorlinks=true, linkcolor=blue, citecolor=red,
urlcolor=blue]{hyperref}
\usepackage[
backend=biber,
style=authoryear, 
citestyle=authoryear-comp,
maxnames=7,
sortcites=false %%%% to keep the order in \cite command
]{biblatex}
\usepackage{csquotes}

\newcommand{\Cl}{\mathrm{Cl}}
\newcommand{\Aut}{\mathrm{Aut}}
\newcommand{\Surj}{\mathrm{Surj}}
\newcommand{\Hom}{\mathrm{Hom}}
\newcommand{\Hur}{\mathrm{Hur}}
\newcommand{\Pic}{\mathrm{Pic}}
\newcommand{\GL}{\mathrm{GL}}
\renewcommand{\Pr}{\mathrm{Pr}}
\newcommand{\Sel}{\mathrm{Sel}}
\newcommand{\Gal}{\mathrm{Gal}}
\newcommand{\im}{\mathrm{im}}
\newcommand{\rank}{\mathrm{rank}}
\newcommand{\F}{\mathbf{F}}
\newcommand{\Z}{\mathbf{Z}}
\newcommand{\Q}{\mathbf{Q}}
\newcommand{\R}{\mathbf{R}}
\newcommand{\C}{\mathbf{C}}
\renewcommand{\P}{\mathbb{P}}
\newcommand{\ra}{\rightarrow}
\newcommand{\inj}{\hookrightarrow}
\newcommand{\OO}{\mathcal{O}}
\newcommand{\FF}{\mathcal{F}}
\newcommand{\EE}{\mathbb{E}}
\newcommand{\set}[1]{\{#1\}}

\newcommand{\tensor}{\otimes}

\newtoggle{tnbcbx@howcited}
\DeclareEntryOption[boolean]{howcited}[true]{\settoggle{tnbcbx@howcited}{#1}}
\DeclareBibliographyOption[boolean]{howcited}[true]{\settoggle{tnbcbx@howcited}{#1}}
\DeclareTypeOption[boolean]{howcited}[true]{\settoggle{tnbcbx@howcited}{#1}}

\newbibmacro{howcited}{%
  \iftoggle{tnbcbx@howcited}
    {\iffieldundef{shorthand}
       {}
       {\setunit{\addspace}%
        \printtext[parens]{%
          \bibstring{citedas}%
          \setunit{\addcolon\space}%
          \printfield{shorthand}%
          \setunit{\addslash}%
          }}}
    {}}

\renewbibmacro{finentry}{\usebibmacro{howcited}\finentry}

\DefineBibliographyStrings{french}{citedas    = {cité ci-dessus avec l'acronyme}}
\DefineBibliographyStrings{english}{citedas    = {heretofore cited as}}

\newcommand{\citelong}[1]{\citeauthor{#1}, \citeyear{#1}}
\newcommand{\textcitelong}[1]{\citeauthor{#1} (\citeyear{#1})}

\DefineBibliographyExtras{french}{\restorecommand\mkbibnamefamily}

\DeclareDelimFormat{nameyeardelim}{\addcomma\space}

\DeclareNameAlias{sortname}{given-family}

\renewcommand{\bibnamedash}{\leavevmode\raise3pt\hbox to3em{\hrulefill}\space}

\AtEveryBibitem{%
  \clearfield{issn} % Remove issn
  \clearfield{isbn} % Remove isbn
  \clearfield{doi} % Remove doi
  \clearlist{language} %%% remove language
  \ifentrytype{online}{}{% Remove url except for @online
  \ifentrytype{unpublished}{}{% ou bien @unpublished
    \clearfield{url}
  }
  }
}

\renewbibmacro{in:}{%
    \ifentrytype{article}{}{\printtext{\bibstring{in}\intitlepunct}}}

\DeclareFieldFormat[article,periodical,inreference]{number}{\mkbibparens{#1}}
\DeclareFieldFormat[article,periodical,inreference]{volume}{\mkbibbold{#1}}
\renewbibmacro*{volume+number+eid}{%
    \printfield{volume}%
    \setunit*{\addthinspace}% NEW
    \printfield{number}%
    \setunit{\addcomma\space}%
    \printfield{eid}}

\DeclareFieldFormat[article,inbook,incollection]{title}{\enquote{#1}\addcomma} 

\date{Mars 2026}
\bbkannee{78\textsuperscript{e} année, 2025--2026}  % Commandes spéciales
\bbknumero{1251}                                      % Commandes spéciales

\title{Recent Progress around Cohen--Lenstra Heuristics}
\subtitle{}

\author{Jordan Ellenberg}
\address{University of Wisconsin-Madison}
\email{ellenber@math.wisc.edu}

\begin{document}

\maketitle

%\section*{Introduction}

\section{Introduction}

The {\em Cohen-Lenstra heuristics} are a family of conjectures in number theory first formulated more than forty years ago in \textcite{cohenlenstra} by Henri Cohen and Hendrik Lenstra.  The conjectures, like so many other fundamental ideas in the history of number theory, were inspired by observations from experimental data.  For instance, Cohen and Lenstra observe that

\begin{quote}
    ``If $p$ is a small odd prime, the proportion of imaginary quadratic fields whose class number is divisible by $p$ seems to be significantly greater than $1/p$ (for instance $43\%$ for $p = 3$, $23.5\%$ for $p=5$).''
\end{quote}

What was novel, in 1983, is that much of the experimental data underlying the heuristics was obtained by machine computation, e.g. Duncan Buell's computation in \textcite{buell} of a million or so class groups of quadratic imaginary number fields, using FORTRAN code on a IBM 370/158.  (This computation, using modern methods, takes 16.3 seconds on a laptop.)

Cohen and Lenstra, in the first sentence of their paper, say of their conjectures that ``proofs seem out of reach at present,'' and that, by and large,  remains the state of things four decades later.  But the conjectures launched a remarkable body of work, serving as the foundation for a broad suite of more general heuristics; and in the last five years, we have seen rapid progress towards proofs of the conjectures in a much broader range of contexts than was previously attainable.  Our goal in the present exposition is to describe some of these current advances and to place them in the context of the broad world of heuristics of Cohen-Lenstra type, which I think are by now mature enough to be called ``the Cohen-Lenstra philosophy.''  Our discussion will center on three main advances: \parencite{lwzb} on the generalized Cohen-Lenstra-Martinet conjectures and their geometric analogues; \parencite{smith1,smith2}, which resolves the $2$-primary Cohen-Lenstra problem for both class groups of quadratic fields and Selmer groups in quadratic twist families; and the proof of Stevenhagen's conjecture on the negative Pell equation by \parencite{koymanspagano}.

\section{The Cohen-Lenstra philosophy}

We begin with the basics.  Let $K$ be a number field.  The {\em class group} $\Cl_K$ is defined to be the group of fractional ideals of $\OO_K$ modulo the group of principal fractional ideals.  The fact that this group is finite, a theorem of Dirichlet, is one of the foundation blocks of algebraic number theory; and this finite abelian group is thus an invariant of $K$ of central importance.  Its order is called the {\em class number} of $K$ and is denoted $h_K$.  

When $K$ is a quadratic imaginary field, the Dirichlet class number formula relates $h_K$ to a special value of an $L$-function; this makes the {\em size} of the class group an invariant that can be approached by analytic means.\footnote{To be more precise, it is the combination of $h_K$ and the {\em regulator} which has an analytic meaning; it is natural to consider these together, and there is another line of work in which one considers not only the class group but an ``Arakelov class group'' whose component group is $\Cl_K$ and whose identity component manifests the regulator; we will not explore this further here, but see e.g. \parencite{bjl}.}  The Brauer-Siegel theorem tells us that, in the quadratic imaginary case $K = \Q(\sqrt{-d})$, one has $\lim \log h_K / \log d = 1/2$ for any infinite sequence of $K$; so the asymptotic behavior of $h_K$ is roughly understood, although many questions certainly remain open.

The {\em structure} of $\Cl_K$ is a completely different story, and this is where the Cohen-Lenstra heuristics enter.  If $N$ is a positive integer, the $N$-torsion subgroup $\Cl_K[N]$ is a finite abelian group killed by $N$.  While the size of $\Cl_K$ grows with $K$, the same need not be the case for $\Cl_K[N]$; indeed, we do not even know whether $\dim_{\F_3} \Cl_K[3]$ is bounded as $K$ ranges over all quadratic imaginary fields (though as we will see below, we expect that it is not) and no quadratic field is known with $\dim_{\F_3} \Cl_K[3] > 8$.\footnote{As far as I know.  Elkies describes $8$ as the record in an email to the NMBRTHRY listserv on 5 Feb 2016.}

Cohen and Lenstra's invocation of the ``proportion of imaginary quadratic fields whose class number is divisible by $p$" requires a little explanation; the set of imaginary quadratic fields is countably infinite, so one has to be careful about what is meant by a ``random" such field.  Following Cohen and Lenstra, we interpret such statements in terms of averages over larger and larger ``boxes," as follows.

Choose an odd prime\footnote{We use $\ell$ rather than $p$ in this setting so that $p$ is available to be the characteristic when we discuss the function field case.} $\ell$ and write $\Cl_K[\ell^\infty]$ for the $\ell$-primary part of $\Cl_K$.  For any large real number $M$, we can define a random variable $X_M$ by choosing a quadratic imaginary field $K = \Q(\sqrt{-d})$ uniformly at random from the fundamental discriminants $d < M$, and then taking $X_M$ to be $\Cl_K[\ell^\infty]$.  This random variable provides a probability distribution $p_M$ on the isomorphism classes of finite abelian $\ell$-groups.  The Cohen-Lenstra philosophy is that the distributions $p_M$ converge to a limiting distribution $p$ as $M \ra \infty$, and that this distribution should be, in some sense, a natural one. We now describe the limiting distribution Cohen and Lenstra proposed in three ways.

\begin{enumerate}
\item (random object) For any finite abelian $\ell$-group $A$, the probability $p(A)$ is proportional to $|\Aut(A)|^{-1}$.  (This very familiar weighting can be thought of as saying that $A$ is an object chosen uniformly at random from the category of finite abelian $\ell$-groups with isomorphisms.)
\item (random equalizer) $p(A)$ is the limit, as $n \ra \infty$, of the probability that a uniformly random matrix $\gamma$ in $M_n(\Z_\ell)$ has $\Z_\ell^n / (\gamma - 1) \Z_\ell^n \cong A$.  This point of view was introduced in \parencite{friedmanwashington}.
\item (moments) If $A$ is a finite abelian $\ell$-group drawn from $p$, and $B$ is any fixed finite abelian $\ell$-group, the expected number of surjections from $A$ to $B$ is $1$.
\end{enumerate}

Each of these three descriptions has some implicit facts embedded in it.  For the first description, one needs the fact that the sum of $|\Aut(A)|^{-1}$ over all isomorphism classes of finite abelian $\ell$-groups is finite; indeed, it is an enjoyable exercise to check that it is 
\begin{equation}
\sum_A |\Aut(A)|^{-1} = \prod_{i \geq 1} (1 - \ell^{-i})^{-1}. \label{eq:sumaut}
\end{equation}

For the second, one needs the fact from the theory of random matrices that the specified limit actually exists.  For the third, one needs to know that specifying those expected values actually specifies the distribution $p$.  The expected value of $|\Surj(A,B)|$ where $A$ is a group drawn from a probability distribution is often called the {\em $B$'th moment} of that distribution, based on the following analogy; if $X$ is a random variable valued in finite sets, and $[k]$ is the set $\set{1,\ldots,k}$, then $|\Hom(X,[k])| = k^{|X|}$, so the expected value of $|\Hom(X,[k])|$ is a (non-integral) moment of the random variable $\exp(|X|)$ in the usual sense.\footnote{Restricting to surjections gives a different sequence of moments which carries the same information as the moments from Hom.}      As in more standard settings, it is an interesting and important question to understand when a distribution is determined by its moments; \textcite{wood:moments} provides an excellent survey of our state of knowledge about this problem for random variables valued in groups (though, as you will see, this state is still continuously developing.)

All three of these descriptions have been of use in various problems in the Cohen-Lenstra domain.  It is certainly not obvious that they give the same answer in the case currently under discussion.  Nor is it obvious from any of these descriptions what the numerical value $p(A)$ is for any particular finite abelian $\ell$-group $A$.  But from the computation \eqref{eq:sumaut}, we see that
$$
p(A) = |\Aut(A)|^{-1} \prod_{i \geq 1} (1 - \ell^{-i}). 
$$

In particular, the probability that $\Cl_K[\ell^\infty]$ is trivial -- equivalently, that $h_K$ is indivisible by $\ell$ --  is $\prod_{i \geq 1} (1 - \ell^{-i})$, which, as Cohen and Lenstra observed, is noticeably smaller than $1-1/\ell$.  When $\ell = 3$, the product above is about $0.56$, in good agreement with the probability of $3$-indivisibility of the class number observed by Cohen and Lenstra. 

In the other direction, Cohen-Lenstra heuristics predict that the probability that a quadratic imaginary field has $\dim_{\F_\ell} \Cl_K[\ell] = r$ is of order $\ell^{-r^2}$.  This rapid decrease explains why class groups are so often cyclic, and why class groups with high $\ell$-rank are in practice transuranically rare.

\section{Generalizations and related problems}

The Cohen-Lenstra philosophy has by now expanded and generalized in a wide range of directions.  In any context where one can define a group $G$ attached to an arithmetic object $X$, and where the arithmetic objects form a countably infinite set which is naturally ordered by some notion of height, one can consider the distribution of the isomorphism classes of $G_X$ as $X$ ranges over objects of height at most $M$.  The corresponding ''Cohen-Lenstra" problem is then to show that these probability distributions approach a limit, and to describe that limiting distribution, which one expects to be in some sense natural.

Even in their original formulation, the Cohen-Lenstra conjectures weren't intended to apply only to quadratic imaginary number fields; Cohen and Lenstra also propose heuristics for the class groups of totally real $A$-extensions of $\Q$, where $A$ is a fixed finite abelian group.  Shortly afterwards,  \parencite{cohenmartinet} proposed a whole family of conjectures, generally known as the Cohen-Lenstra-Martinet heuristics, which concern the variation of the $S$-primary part of class group for any set $S$ of primes, extensions of any degree and indeed any specified Galois group, over an arbitrary base number field $F$.  Over time, many subtleties emerged and it became clear that the original heuristics, while sound in spirit, were incorrect in a number of particulars.  In recent years, the situation has improved markedly, to the point where we now have fully satisfactory heuristics for the class groups of extensions of number fields with specified Galois group.  We report on this progress in Section~\ref{s:clm}.

One can also go beyond class groups.  For instance, the Selmer group of a random elliptic curve, or even a random abelian variety, over a number field $F$, has much in common with a class group, and indeed, \parencite{bklpr} formulates a conjecture for the variation of Selmer groups in families of abelian varieties which parallels the Cohen-Lenstra philosophy. We'll discuss some progress towards these conjectures in \S~\ref{s:smith}.

In another direction, one can ask about the maximal everywhere unramified extension in toto instead of the maximal abelian unramified extension; we'll discuss this in \S~\ref{s:nonabelian}.

The right level of generality for the base field $F$ should include not just number fields but general global fields; that is, we should admit the case where $F$ is the function field of a curve over a finite field.  Studying such cases brings Cohen-Lenstra problems in contact with questions in algebraic geometry and even topology; the function field case has enjoyed rapid recent progress lately, most notably in \parencite{landesmanlevy}, and has been a source of insight about Cohen-Lenstra problems for quite some time, and we will try to keep both cases in full view throughout our discussion.

One situation that has remained mostly mysterious is that of the $p$-part of the class group for extensions of a global field of characteristic $p$.  Here, the natural object of interest is not really a group at all, but a group {\em scheme} arising from the $p$-torsion of an abelian variety in characteristic $p$.  Little is known about this, but the very interesting recent result of \parencite{gartonthunder} about the $3$-torsion in hyperelliptic Jacobians in characteristic $3$ is worth noting here.

\section{Cohen-Lenstra-Martinet heuristics over general global fields: work of Liu, Sawin, Wang, Wood, Zureick-Brown, etc.}

\label{s:clm}

The Cohen-Lenstra-Martinet heuristics concern the following situation. Let $F$ be a number field, let $\Gamma$ be a finite group, let $S$ be a set of primes not including any divisors of $|\Gamma|$, and write $\Z_{(S)}$ for the localization of $\Z$ at $S$.  For every Galois extension $K/F$, the class group $\Cl_K \tensor_\Z \Z_{(S)}$ is a $\Z_{(S)}[\Gamma]$-module.  One may then ask:  what is the distribution on the isomorphism class of $\Cl_K \tensor_\Z \Z_{(S)}$ when $K$ is a random $\Gamma$-extension of $F$?

The Cohen-Lenstra philosophy might be taken to suggest that the answer should be ``a module $A$ appears with probability inversely proportional to the size of its automorphism group.''  Immediately a subtlety arises;  its automorphism group {\em as what?}  To define $\Aut(A)$ requires a choice of category in which $A$ resides as an object.

In this case, one might naturally predict that a module $A$ will appear with probability inversely proportional to the number of automorphisms of $A$ as a $\Z_{(S)}[\Gamma]$-module. This turns out to be too na\"ive.  In some respects, this is obvious; for instance, the $\Gamma$-invariants of $\Cl_K \tensor_\Z \Z_{(S)}$ are just $\Cl_F \tensor_\Z \Z_{(S)}$, which is fixed; so we may restrict our attention to those $A$ which have the right space of $\Gamma$-invariants (in particular, which have trivial $\Gamma$-invariants in the case $F=\Q$).

Even taking this into account, the naively applied heuristics give answers that clearly deviate from experiment, even in the case of real quadratic fields, as Cohen and Lenstra observed.  Numerical evidence suggests that, as $K$ ranges over real quadratic extensions of $\Q$, the probability that $\Cl_K[\ell^\infty]$ is isomorphic to $A$ is proportional, not to $|\Aut(A)|^{-1}$, but to $|A|^{-1} |\Aut(A)|^{-1}$.

This observation can be brought in line with the general philosophy in several ways.  \parencite{cohenlenstra} records the observation, credited to Gross, that when $K$ is a real quadratic field, the ring $\OO_K$ is in some sense analogous to the ring $\OO_L[1/\pi]$, where $L$ is a quadratic imaginary field and $\pi$ is a factor of a prime $p$ split in $L$.  Both rings have two ``missing'' places: the two archimedean places in the case of $\OO_K$, and $\infty$ and $\pi$ in the case of $\OO_L[1/\pi]$.  The class group of $\OO_L[1/\pi]$ is simply $\Cl_L / \langle \pi \rangle$.  This suggests that $\Cl_K$ should be seen, not as a random abelian $\ell$-group, but as the quotient of a random abelian $\ell$-group (like $\Cl_L$) modulo a random element (like $\pi$).  And indeed, one can check that if $B$ is drawn from the Cohen-Lenstra distribution on finite abelian $\ell$-groups, and $b$ is drawn uniformly from $A$, then the probability that $B/\langle b \rangle \cong A$ is proportional to $|A|^{-1} |\Aut(A)|^{-1}$.

Another point of view is offered in \parencite{wangwood} and \parencite{lwzb}.  These papers offer a very satisfactory version of the Cohen-Lenstra-Martinet heuristics for $\Gamma$-extensions of $\Q$, and provide compelling geometric evidence for its truth, as we now explain.

First of all, \parencite{wangwood} show that the Cohen-Lenstra-Martinet heuristics, slightly modified and constrained, can be phrased in the following compact way.   Let $S$ be a finite set of primes not containing any divisors of $2|\Gamma|$. Write $\Cl^S_K$ for $\Cl_K \tensor_\Z \Z_{(S)}$, which is a finite $\Z_{(S)}[\Gamma]$-module (that is, an abelian $S$-primary group endowed with an action of $\Gamma$). Let $H$ be a finite abelian $\Z_{(S)}[\Gamma]$-module and suppose that $H^\Gamma$ is trivial.  Let $\Gamma_\infty$ be a subgroup of $\Gamma$ of order $1$ or $2$ (defined up to conjugacy in $\Gamma$), and denote by $\FF(\Gamma_\infty,X)$ the set of Galois extensions $K/\Q$ with Galois group $\Gamma$ such that $\Gamma_\infty \subset \Gamma$ is a decomposition group at an archimedean place of $K$, and such that the product of the ramified primes in $K/\Q$ is at most $X$. 

\begin{conj}
There is a real number $c$ (depending on $\Gamma, \Gamma_\infty, S$) with the following property: for every finite $\Z_{(S)}[\Gamma]$-module $H$, the probability that an element $K/\Q$ of $\FF(\Gamma_\infty,X)$ has $\Cl^S_K \cong H$ approaches 
$$
c |H^{\Gamma_\infty}|^{-1} |\Aut_\Gamma(H)|^{-1}
$$
as $X \ra \infty.$
\label{co:wwclm}
\end{conj}

(The constant $c$ can be computed explicitly as the reciprocal of the sum of $|H^{\Gamma_\infty}|^{-1} |\Aut_\Gamma(H)|^{-1}$ over isomorphism classes of $\Z_{(S)}[\Gamma]$-modules; in other words, one is also conjecturing that there is no ``escape of mass" as $X$ grows.)

This version of the heuristic avoids many of the known pitfalls of the Cohen-Lenstra-Martinet conjectures as originally formulated, which are illustrative of the subtleties of this topic.  We address a few of these now.

First of all, Wang and Wood require that $S$ be a finite set, so that we are considering only finitely many primes at a time.  The necessity of placing some restrictions on the infinite $S$ case was already observed in \parencite[\S 6]{bartellenstra}, and the correct conjecture to make in this general setting remains somewhat mysterious.  One does not want to throw up one's hands totally and remain silent on all questions involving infinitely many primes; for instance, Cohen and Lenstra's original prediction that about three-quarters of all real quadratic fields with prime discriminant have class number $1$ is of this form (since it places a restriction on the $\ell$-primary part for {\em every} odd $\ell$) and is still generally believed.  But we'll say no more about it in these notes, and restrict to finite $S$ hereafter.

Another change is the order in which the $K$ are counted.  When $K$ is quadratic imaginary, there is a natural ordering that presents itself; we count the fields $K$ in order of the absolute value of their discriminant $d_K$.  For general $F$ and any fixed value of $[F:K]$, there are only finitely many fields $K/F$ whose discriminant has absolute norm at most $X$.  Thus the discriminant provides a natural order in which to count number fields.  Indeed, for many results that use the theory of prehomogeneous vector spaces in the style of Bhargava and his collaborators, one has little choice but to count fields in order of discriminant.  But it has come to be seen that this ordering is problematic for Cohen-Lenstra heuristics.  Roughly speaking, the problem is subfields. If one counts $\Z/4\Z$-extensions $K$ of $\Q$ in order of discriminant, for example, and $L$ is a fixed quadratic field, then a positive proportion of $K$ contain $L$, and so any unexpected behavior of a single $\Cl_L$ can bias the distribution of $\Cl_K$ over all $K$.  \parencite[\S 6]{bartellenstra} shows that indeed this phenomenon produces counterexamples to the original Cohen-Lenstra-Martinet conjecture when fields are counted by discriminant.  Counting by product of ramified primes appears, so far, to have much ``smoother" statistical properties, a fact first observed in \parencite{wood2010}.  It is suggested in \parencite{sawinwoodroots}[Remark 1.3] that any natural ordering on extensions $K/F$ will do as long as no intermediate subfield occurs a positive proportion of the time.

One may still well ask: what happened to the philosophy that every group should occur inversely proportionally to its number of automorphisms?  The statement of Conjecture~\ref{co:wwclm} doesn't quite match that expectation.  This is explained by \cite{wangwood}[Theorem 5.1].  In fact, it turns out that $\Cl_K^S$ is not just a $\Gamma$-module, but rather carries a slightly more refined structure called a {\em class triple}; the class triple is determined up to isomorphism by the isomorphism class of $\Cl_K^S$ as $\Gamma$-module, but has fewer automorphisms.  Indeed, the size of its automorphism group is inversely proportional to the probability in Conjecture~\ref{co:wwclm}.  This kind of thing is not infrequent in this part of mathematics; an entity does not have a well-defined automorphism group until we understand what category we're thinking of the entity as an object of, and that choice may be quite subtle!

\subsection{The method of moments}

A proof of Conjecture~\ref{co:wwclm} still seems far out of reach.  But \parencite{lwzb} provides some compelling evidence for its correctness, as we now explain.

The argument passes through the method of moments. Recall what we mean by moments in this context:

\begin{defi} Let $R$ be a ring and $X$ a random variable valued in finitely generated $R$-modules, and $Y$ a finite $R$-module.  Then the $Y$-moment of $X$ is the expected value $\EE|\Surj(X,Y)|$ of the number of surjective $R$-module homomorphisms from $X$ to $Y$.
\label{de:moments}
\end{defi}

For example, if $X$ is a random variable valued in finitely generated $\Z_\ell$-modules, the $(\Z/\ell\Z)$-moment of $X$ is the expected value of $\ell^{r(X)}-1$, where $r(X)$ is the rank $\dim_{\F_\ell} X / \ell X$. 

For class groups of random number fields, proving theorems about moments is often easier than proving theorems about distributions.  The Davenport-Heilbronn theorem, for example, shows that the $(\Z/3\Z)$-moment of the class group of a random quadratic imaginary field $K$ is $1$.  No other moments of the distribution of $\Cl_K[3^\infty]$ as $K$ ranges over quadratic imaginary fields are known.  Many of the breakthrough results of the Bhargava school are also computations of individual moments of random class groups or random Selmer groups.

The moments of the Cohen-Lenstra-Martinet distribution proposed in Conjecture~\ref{co:wwclm} are computed in \parencite{wangwood}.

\begin{theo}[Thm 6.1,\parencite{wangwood}]
If $X$ is a random $\Z_{(S)}[\Gamma]$-module drawn from the distribution in Conjecture~\ref{co:wwclm}, then for every finite $\Z_{(S)}[\Gamma]$-module $H$,
\begin{equation}
    \EE |\Surj(X,H)| = |H^{\Gamma_\infty}|^{-1}. \label{eq:clmmoments}
\end{equation}
\label{th:clmmoments}
\end{theo}

One naturally wonders about a converse to the above theorem; if \eqref{eq:clmmoments} holds for all $H$, does that mean $X$ obeys the conjectured distribution?

In traditional probability theory, it's a foundational question to understand whether there exists a probability distribution $\nu$ on the real numbers giving rise to a specified sequence of moments $a_k = \int_\R x^k d \nu$, and if so whether $\nu$ is the unique such probability distribution.  By now, much is understood about these questions; for example, under the condition that the $a_k$ don't grow too fast, such a measure exists if and only if a certain positivity criterion is satisfied, and when it exists, it is unique.  In particular, under a growth condition, knowing all the moments determines the distribution.  In the context of random abelian groups arising from arithmetic, one sees ``moments determine distributions" techniques arising as early as \parencite{heathbrown:congruent2}.

Extending these results to moments in the more general categorical sense considered here (and beyond) has been a topic of much recent interest, culminating in the results of \parencite{sawinwoodmoments}, which proves a remarkable ``method of moments" result for random variables valued in objects of suitable categories called {\em diamond categories}.  This class of categories includes not only finite $R$-modules, the case considered here, but finite groups, finite commutative rings, finite Lie algebras, and many more exotic examples besides.  (For instance, the argument of \parencite{sawinwood3manifolds} on the profinite fundamental group of a random 3-manifold requires an argument of this kind for the category of pairs $(G,h)$ where $G$ is a finite group and $h$ is an element of $H_3(G,\Z)$!)
Sawin and Wood's main theorem is quite analogous to the situation for real-valued variables; as long as the specified moments ``grow slowly enough," there is at most one distribution with those moments; such a distribution exists if and only if the moments satisfy a positivity criterion; and, when this criterion is satisfied, the distribution can be explicitly described.

So the situation would be completely satisfactory if we knew all the moments of the distribution of $\Cl^S_K$ as $K$ ranges over $\Gamma$-extensions of $\Q$.  We do not know all of them, and in fact, apart from a few exceptional cases, do not know any of them.  But there is a closely allied question in arithmetic geometry where much more is known. We turn to this now.

\subsection{Moments and limits in the function field case}

\label{s:ffmoments}

All the questions discussed so far can be asked when the base field $F$, rather than being a number field, is the function field of a curve over a finite field $\F_q$.  (In this case, we add to our running list of constraints the additional one that $|\Gamma|$ is prime to $q$, and $S$ does not contain the characteristic.)  When a heuristic is meant to hold for all number fields $F$, one naturally suspects that it should hold for function fields as well.  What's more, the function field version of the heuristic often has geometric content that's invisible or obscure on the number field side, and the geometry may make it easier to actually prove things rather than merely conjecture.  

So it has transpired in the present context.  We take $F$ to be the rational function field $\F_q(t)$. The object of study is now the arithmetic of a random $\Gamma$-extension $K / F$, the product of whose ramified primes has norm at most $X$. Such a $K$ is the function field of a smooth curve $Y/\F_q$ with a $\Gamma$-action and an identification of $Y/\Gamma$ with $\P^1/\F_q$, such that the branch locus $B$ in $\P^1$ of the map $Y \ra \P^1$ satisfies $q^{\deg B} < X$. For simplicity, we may take $X$ to be a power $q^n$ of $q$ and consider only those covers whose branch locus on $\P^1$ has degree {\em exactly} $n$, rather than at most $n$.  Such curves are parametrized by the $\F_q$-rational points of a moduli space called a {\em Hurwitz space}, which we denote $\Hur_{\Gamma, n} / \F_q$. 

The notion of ``$\Cl_K$'' now requires a bit more care.  When $K$ is a number field, this means $\Pic(\OO_K)$, where $\OO_K$ is the integral closure of $\Z$ in $K$.  In the function field case, there is no such canonical affine ring.  However, if we choose a copy of $\F_q[t]$ in $\F_q(t)$ (equivalently: choosing a rational point called $\infty$ in $\P^1$, and identifying $\F_q[t]$ with the coordinate ring of $\P^1 - \infty$) we can define $R_K$ to be the integral closure of $\F_q[t]$ in $K$, and $\Cl_K$ to be $\Pic(R_K)$.  This is the approach taken in \parencite{lwzb}; by analogy with the number field case, they say $K$ is {\em totally real} if $Y$ is split completely over $\infty$.  In this case $\Cl_K$ is the quotient of $\Pic(Y)(\F_q)$ by the subgroup generated by the $|\Gamma|$ points of $Y$ lying over $\infty$. 

For each such $K$, we are asking about the number of surjective homomorphisms from $\Cl^S_K$ to some fixed $\Z_{(S)}[\Gamma]$-module $H$.  Such a homomorphism yields a covering $Y' \ra Y$ with Galois group $H$, which is unramified except possibly over $\infty$.  One is thus naturally led to the moduli space, which for the moment we'll call $\Hur^H_{\Gamma,n}$, whose points parametrize $\Gamma$-covers $Y \ra \P^1$ together with an unramified $H$-cover $Y' \ra Y$ such that $Y' \ra \P^1$ is Galois and the action of $\Gamma$ on $H$ is the desired one.

All this being done, the number of $\Gamma$-extensions $K$ being counted has been expressed as $|\Hur_{\Gamma,n}(\F_q)|$, and the total number of surjections from $\Cl^S_K$ to $H$ afforded by all these $K$ together is $|\Hur^H_{\Gamma,n}(\F_q)|$.  The $H$-moment we aim to compute is then nothing other than the ratio
\begin{equation}
\frac{|\Hur^H_{\Gamma,n}(\F_q)|}{|\Hur_{\Gamma,n}(\F_q)|}
\label{eq:hurratio}
\end{equation}
and the heuristic prediction would be that this ratio converges to a limit as $n \ra \infty$, and furthermore that the limit is equal to some specified predicted value.  

I have explained this part rather quickly because this approach to the Cohen-Lenstra conjecture over rational function fields has been discussed before in this seminar, in \parencite{orwbourbaki}, in connection with earlier work in \parencite{evw} that used these ideas to prove results towards the function field Cohen-Lenstra heuristics in the quadratic case $(|\Gamma| = 2)$.

What can be done nowadays is far more general.  For instance, while $|\Hur_{\Gamma,n}(\F_q)|$ and $|\Hur^H_{\Gamma,n}(\F_q)|$ may be hard to compute, it is much easier to understand their behavior as $q \ra \infty$ with $n$ fixed; by the Weil conjectures, this only depends on the number of irreducible components of these spaces, their dimensions, and the action of Frobenius on them.  In some favorable cases, both spaces are geometrically irreducible, in which case the ratio \eqref{eq:hurratio} tends to $1$ as $q \ra \infty.$  In general, the description of the irreducible components is substantially more subtle; but it is now completely understood, thanks to the results of \parencite{wood:liftinginvariant}.

From here, the strategy proceeds as follows.  Write $p_{q,n}$ for the probability distribution on isomorphism classes of $\Z_{(S)}[\Gamma]$-modules given by $\Cl^S_K$, where $K$ is chosen uniformly from the (finite) set of $\Gamma$-extensions $K/\F_q(t)$ whose branch locus on $\P^1$ has degree $n$.  Write $M_{H,q,n}$ for the expected size of $|\Surj_\Gamma(\Cl^K_S,H)|$ as $K$ ranges over this finite set; in other words, $M_{H,q,n}$ is the $H$-moment of the distribution $p_{q,n}$.  The understanding of the irreducible components of Hurwitz spaces allows Liu, Wood, and Zureick-Brown to compute $\lim_{q \ra \infty} M_{H,q,n}$ with $n$ fixed, since the behavior of \eqref{eq:hurratio} as $q \ra \infty$ is governed by a computation of these components.  They arrive at the following appealing conclusion:

\begin{prop}[\parencite{lwzb}, Corollary 1.5] For every finite $\Z_{(S)}[\Gamma]$-module $H$,
$$
\lim_{n \ra \infty} \lim_{\substack{q\to \infty\\ (|H|,q-1)=1}} M_{H,q,n} = |H|^{-1}.
$$
\label{pr:lwzbmoments}
\end{prop}

In other words, the limit of the moments exactly conforms to the moments of the Cohen-Lenstra-Martinet distribution computed in Theorem~\ref{th:clmmoments}.

We now come to another probabilistic wrinkle.  We have discussed the question of whether the moments of a probability distribution on groups determine the distribution.  It turns out (by contrast with the classical cases of real-valued random variables) to be quite another thing to say that the limits of the moments of a sequence of distributions determine the weak limit of the distributions. So one doesn't immediately get a statement about a weak limit of the distributions $p_{q,n}$ from Proposition~\ref{pr:lwzbmoments}. The following example (Example 6.14 in \parencite{wangwood}) is illustrative in this respect.  Suppose $X$ is a random variable valued in finite abelian groups.  For each prime $\ell$ write $X_\ell$ for $X \oplus (\Z/\ell \Z)$.  Now for any finite abelian group $A$, the $A$-moment of $X_\ell$ is equal to that of $X$ for all but finitely many $\ell$, since a homomorphism from $X \oplus (\Z/\ell \Z)$ must kill the second factor for $\ell > |A|$.  So the sequence $X_2, X_3, X_5, \ldots$ has the property that the limit of its moment sequence is equal to the moment sequence of $X$.  But this sequence of probability distributions certainly doesn't converge to the distribution of $X$.  Just choose some $A$ such that $\Pr(X \cong A) > 0$, and note that $\Pr(X_\ell \cong A) = 0$ for $\ell$ large enough.

This is precisely where the constraint that $S$ is finite becomes important.  The above example relies on the fact that the random variables $X_\ell$ are valued in finite abelian groups, but not in finite abelian $S$-primary groups for any fixed finite $S$. The results of \parencite{wangwood} (and the more general results of the same nature in \parencite{sawinwoodmoments}) show, that in the setting considered here, where $S$ is finite, the solution to the moment problem is not only unique but, in their notation, {\em robustly unique}.  That means that if a sequence of distributions on $\Z_{(S)}[\Gamma]$-modules has the property that its $A$-moment converges to some $m_A$ for every $A$, the sequence of distributions itself converges weakly to the unique distribution with moments $m_A$.  In particular, we can conclude that, for any $A$, the weak limit $\lim_{n \ra \infty} \lim_{q \ra \infty} p_{q,n}(H)$ is the distribution predicted in Conjecture~\ref{co:wwclm}.

This does not prove the Cohen-Lenstra-Martinet conjecture over function fields, which would predict that, for every $q$, the distributions $p_{q,n}$ weakly converge to the C-L-M distribution as $n \ra \infty$.  But if you believe $\lim_{n \ra \infty} p_{q,n}$ exists and does not depend on $q$, then this result does single out the C-L-M distribution as the only reasonable expectation for that limit.

To get the Cohen-Lenstra-Martinet conjecture over $\F_q(t)$, we would need to know that the limiting moment $\lim_{n \ra \infty} M_{H,q,n}$ exists and is equal to $|H|^{-1}$ for every $H$.  While this remains unknown, a major step in this direction has recently been announced in~\parencite{landesmanlevy}, which proves that $\lim_{n \ra \infty} M_{H,q,n}=|H|^{-1}$ for all $q$ which are sufficiently large relative to $H$.  The Landesman-Levy argument is primarily topological in nature and involves lifting the Hurwitz space to characteristic $0$ and studying its homology.  From $H$'s point of view, this is almost a complete resolution of the problem; all but finitely many $q$ are covered!  From $\F_q(t)$'s point of view, we're much farther from the finish line, since one can handle only finitely many $H$. Pushing this kind of argument further to give infinitely many moments for a fixed $q$ remains a major open problem in the subject.

\section{What about $\ell=2$?}

Throughout the paper so far, we have been focusing on the odd parts of class groups.  It is time to explain why.

There is, first of all, the issue of roots of unity. It turns out that the distribution of $\Cl_K[\ell^\infty]$ is sensitive to the presence of $\ell$-power roots of unity in the base field $F$, even in the case of quadratic extensions $K/F$; this phenomenon was first observed in \textcite{malleroots}, which includes extensive computations of $\Cl_K[3]$ as $K$ ranges over quadratic extensions of $\Q(\zeta_3)$. (Malle notes that a related observation appears even earlier, in \cite{gerthillinois}.) This is why the condition $(|H|,q-1) = 1$ appears in Proposition~\ref{pr:lwzbmoments}; it exactly guarantees that there is no nontrivial root of unity in $\F_q(t)$ whose order divides that of $|H|$.  Indeed, the function field case gives good intuition for why roots of unity matter.  Suppose we model the $\ell$-torsion in a class group as $\Pic(X)(\F_q)[\ell]$, which is the $1$-eigenspace for Frobenius acting on $\Pic(X)[\ell](\overline{\F}_q)$.  The Weil pairing identifies the $1$-eigenspace with the dual of the $q$-eigenspace.  This does not supply $\Pic(X)(\F_q)[\ell]$ with any extra structure, unless $q$ is congruent to $1$ mod $\ell$, in which case $\Pic(X)(\F_q)[\ell]$ is self-dual, and in particular, since the pairing is alternating, must have even dimension.  In random matrix terms, one may compute the cokernel of $\gamma-1$ where $\gamma$ is a random symplectic similitude that multiplies the alternating form by $q$; it turns out that the distribution on cokernels is the Cohen-Lenstra distribution when $q \neq 1$, but something different when $q=1$.  This point of view was used in \parencite{garton} to propose modified versions of the Cohen-Lenstra conjecture in some cases where extra roots of unity were present.

Due to the vagaries of the archimedean place, we do not see strict self-dualities on $\ell$-parts of class groups in the number field case; the number field analogues are usually called {\em reflection principles}, and just as in the function-field scenario above, extra roots of unity tend to place class groups in correspondence with themselves under the appropriate such principle.  As Malle observed, this requires modifications to the Cohen-Lenstra-Martinet heuristics, but the precise changes required were not easy to pin down.  The paper \parencite{sawinwoodroots} establishes a convincing version of the Cohen-Lenstra-Martinet heuristics over an arbitrary number field $F$, backed up by function field analogies as in \cite{lwzb}; the presence of roots of unity in $F$, together with the reflection principles they give rise to, is the principal new difficulty.

There's another serious problem with $\ell=2$, which is that the Cohen-Lenstra conjecture as written is catastrophically false in this case.  Suppose $K = \Q(\sqrt{-d})$ is a quadratic imaginary field.  Then if $d'$ is a divisor of $d$ congruent to $1$ mod $4$, $K(\sqrt{d'})$ is an everywhere unramified quadratic extension of $K$.  This argument, which originates with Gauss and which is generally known as {\em genus theory} when done more carefully, means that $\Cl_K[2]$ grows like the number of divisors of $d$, which does not approach a distribution (for instance, the average of the number of divisors of an integer in $[1,X]$ grows without bound as $X \ra \infty$.)  That this does not put a complete end to the study of Cohen-Lenstra at 2 was actually first observed before the Cohen-Lenstra conjectures were formulated, in \parencite{buell}:
\begin{quote}
    ``the $2$-Sylow subgroup tends to be $k - 2$
elementary $2$-groups and one large cyclic factor collecting the other powers of two in
the class number, so that the $2$-Sylow subgroup of the subgroup of squares is cyclic.
In computing the $2$-Sylow subgroup, then, we actually computed that subgroup of the
subgroup of squares, and shall, by abuse of language, call this the $2$-Sylow subgroup.''
\end{quote}

In other words, Buell suggested studying $2 \Cl_K[2^\infty]$ instead of $\Cl_K[2^\infty]$ itself, and that has proven to be a correct impulse.  That $2\Cl_K[2^\infty]$ obeys the $2$-primary Cohen-Lenstra distribution seems to have been first formally conjectured in \parencite{gerth:extension}; earlier, \parencite{gerth:4rank} computed the distribution of the $\F_2$-vector space $2\Cl_K[4]$ as $K$ ranged over quadratic imaginary fields with a fixed number of ramified primes, and \parencite{fouvrykluners} extended this to an average over all quadratic imaginary fields. 

Despite these successes, most people who thought about Cohen-Lenstra and surrounding problems, your correspondent included, thought of $\ell=2$ as a curious side problem, with particular difficulties that made it even less tractable than Cohen-Lenstra in its original form.

This turned out not to be the case.

\section{Variation of $2^\infty$-Selmer groups in quadratic twist families: the work of Smith}

\label{s:smith}

Around ten years ago, Alexander Smith, then a Ph.D. student, began releasing a series of papers which would launch a new direction in the study of Cohen-Lenstra heuristics.  We will focus here mostly on the results of \parencite{smith1,smith2}.  These papers are long and intricate, and involve many different techniques; in this expos\'{e} we will merely try to convey some of the main ideas.

To begin with, Smith does not frame his work as an approach to the variation of class groups, but to the variation of {\em Selmer groups.}  Selmer groups first appear in the study of the arithmetic of abelian varieties over number fields.  If $F$ is a number field, $A/F$ an abelian variety, and $N$ an integer, then the Kummer map provides an injection $\kappa: A(F) / NA(F) \inj H^1(G_F, A[N])$, where $G_F$ denotes the absolute Galois group of $F$ and the $H^1$ is Galois cohomology.  (The basics of Selmer groups and Galois cohomology can be found in the standard reference \parencite{silverman}.) The Kummer map $\kappa$ is very far from surjective, but we can bring the two sides closer by imposing local conditions on the right-hand side which are satisfied by its image.  In particular, we have local Kummer maps
$$
\kappa_v: A(F_v) / NA(F_v) \inj H^1(G_{F_v}, A[N])
$$
for each place $v$ of $F$, and we may define 
$$
\Sel^N(A/F) = \set{\zeta \in H^1(G_F, A[N]): \zeta|G_{F_v} \in \im(\kappa_v), \forall v}.
$$
When $v$ is a non-archimedean place of $F$ which is prime to $N$ and where $A$ has good reduction, $\im(\kappa_v)$ is the set of cohomology classes restricting to $0$ on inertia.  So the Selmer group is a subgroup of $H^1(G_F, A[N])$ cut out by the condition of being unramified at all primes outside some specified finite set $S$, and by some further local conditions at the primes in $S$.  Nowadays, a wide range of groups of this form, with $A[N]$ replaced by this Galois representation or that one, are known as Selmer groups.  In particular, the (Pontryagin dual of the) ideal class group of a number field modulo $N$, which is just the set of classes of $H^1(G_F, \Z/N\Z)$ which are unramified everywhere, is a Selmer group in this general sense. What's more, if $M$ is a $G_F$-module isomorphic as abelian group to $\Z/N\Z$, then $H^1(G_F,M)$ is related to the mod $N$ class group of an extension of $F$ trivializing $M$.

This is the point of view Smith adopts in \parencite{smith1}.  He studies Selmer groups in $H^1(G_F,V)$ where $V$ is what Smith calls \parencite[Definition 4.1]{smith1} a {\em twistable module.}  The precise definition here need not concern us; suffice it to say that Smith works with coefficient  systems $V$ endowed with extra structure such that, for every character $\chi: G_F \ra \pm 1$, one has local conditions cutting out a Selmer group in $H^1(G_F, V^{\chi})$.  This definition is general enough to provide access to both the usual mod $2^k$ Selmer groups of an abelian variety and its quadratic twists, and $\Cl_K / 2^k \Cl_K$ as $K$ ranges over quadratic extensions of $\Q$.

We now state the consequences of Smith's results in the two special cases already mentioned.  We begin with the Cohen-Lenstra conjectures for class groups.  
Some notation:  
\begin{itemize}
    \item 
by $P^{\text{Mat}}(j|n)$ we mean the probability that a random $n \times n$ matrix over $\F_2$ has kernel of dimension exactly $j$.  By $P^{\text{Mat}}(j|\infty)$ we mean the limit of this quantity as $n \ra \infty$.  (That this limit exists is already a contentful fact about random matrix theory, one known prior to the work discussed here.)
\item If $K$ is a number field and $k$ a positive integer, by $r_{2^k}(K)$ we mean the dimension of the $\F_2$-vector space $2^{k-1}\Cl_K[2^k]$.  This list of ranks determines the isomorphism class of $\Cl_K[2^\infty]$.

\end{itemize}
\begin{theo}[\parencite{smith1},Theorem 1.9]
 Given any nonincreasing sequence
\[
r_4 \geq r_8 \geq \dots \geq r_{2^k} \geq \dots
\]
of nonnegative integers, we have
\begin{align*}
&\lim_{H \to \infty} \frac{\# \{ d \in \mathbb{Z}^{>0} : d < H \text{ and } r_{2^k} (\mathbb{Q}(\sqrt{-d})) = r_{2^k} \text{ for } k \geq 2 \}}{H} \\
&\quad = P^{\text{Mat}}(r_4 \mid \infty) \cdot \prod_{k=3}^{\infty} P^{\text{Mat}}(r_{2^k} \mid r_{2^{k-1}}).
\end{align*}
\label{th:smithcg}
\end{theo}

The model to have in mind here is that the $2$-primary class group of a random quadratic imaginary number field behaves like a {\em Markov chain}.  The group $\Cl_{\Q(\sqrt{-d})}[2]$ is governed by genus theory; its rank is roughly the number of prime divisors of $d$, and thus should tend to grow with the size of $d$ on average rather than approaching a distribution.  However, as we have already discussed, the dimension of the kernel of a large random matrix over a finite field converges to a distribution.  What Smith's theorem says (in conformity with Gerth's conjecture, proved by Fouvry and Kl\"{u}ners) is that $2\Cl_{\Q(\sqrt{-d})}[4]$ is drawn from this distribution; having made this draw, $4\Cl_{\Q(\sqrt{-d})}[8]$ is distributed as is the kernel of a random $r_4 \times r_4$ matrix, and so on.  The Markov chain that gives the sequence $r_4, r_8, r_{16}, \ldots$ is nonincreasing and has $0$ as unique absorbing state.

This statement does not on its face resemble the Cohen-Lenstra conjectures as formulated earlier in the paper for odd primes (and extended to the $2$-primary case in \parencite{gerth:extension}).  However, the distribution on the isomorphism class of the finite $2$-primary group $2\Cl_K[2^\infty]$ determined by this Markov process turns out to be the same as that provided by the cokernel of $\gamma-1$ where $\gamma$ is a random matrix in $\GL_N(\Z_2)$ with $N$ very large, which conforms with Gerth's conjecture.  Smith expresses the distribution in this way because it reflects what actually happens in the proof, which relies on an iterative argument computing the distribution of $2^k\Cl_K[2^{k+1}]$ conditional on the isomorphism class of $2^{k-1}\Cl_K[2^k]$.

\begin{rema}
    Not for the last time, I am working in less generality than Smith does.  Many of his techniques and results also apply to $\ell$-primary Selmer groups of cyclic degree-$\ell$ twists; for simplicity of exposition, I am restricting to the case $\ell=2$, where some of the most impressive applications can already be found.  In particular, Smith also shows~\parencite[Theorem 1.12]{smith1} that the distribution of the $\ell^\infty$ class groups of cyclic $\Z/\ell\Z$-extensions of a number field $F$ is as the generalized Cohen-Lenstra heuristics predict -- as long as $F$ does not contain a $2\ell$th root of unity.  So we see once again the presence of roots of unity appearing as technical obstacles in the subject. 
\end{rema}

A similar result holds for quadratic twists of abelian varieties.  Here, $P^{\text{Alt}}(j|n)$ means the probability that a random {\em alternating} $n \times n$ matrix over $\F_2$ has kernel of dimension $j$.  But here some care is needed; since the parity of the kernel of an alternating $n \times n$, matrix is the same as that of $n$, the quantity $P^{\text{Alt}}(j|n)$ does not approach a limit as $n \ra \infty$; we only have a limit if $n$ grows through even integers or through odd integers.  So by $P^{\text{Alt}}(j|\infty)$ we mean the average of these two limits.

When $A$ is an abelian variety, we denote by $r_{2^k}(A)$ the dimension of $2^{k-1} \Sel^{2^k}(A)$ as an $\F_2$-vector space.  If $A$ is an abelian variety over $\Q$ and $d$ is a nonzero integer, then $A^d$ denotes the {\em quadratic twist} of $A$ by $d$, an abelian variety over $\Q$ which becomes isomorphic to $A$ over the quadratic extension $\Q(\sqrt{d})$.  This is very concrete when $E$ is an elliptic curve with equation $y^2 = f(x)$, in which case $E^d$ is simply the elliptic curve with equation $dy^2 = f(x)$.

\begin{theo}[\parencite{smith1},Theorem 1.5, \parencite{smith2}, Theorem 2.14] Suppose $A/\mathbb{Q}$ is an elliptic curve, which satisfies the conditions discussed in Remark \ref{rema:conditionsav} below.  Given any nonincreasing sequence
\[
r_2 \ge r_4 \ge \dots \ge r_{2^k} \ge \dots
\]
of nonnegative integers, we have
\begin{align*}
\lim_{H \to \infty} & \frac{\#\{d \in \mathbb{Z}^{\neq 0} : |d| < H \text{ and } r_{2^k}(A^d) = r_{2^k} \text{ for all } k \ge 1\}}{2H} \\
&= P^{\text{Alt}}(r_2 \,|\, \infty) \cdot \prod_{k=2}^{\infty} P^{\text{Alt}}(r_{2^k} \,|\, r_{2^{k-1}})
\end{align*}
\label{th:smithselmer}
\end{theo}

\begin{rema} Again, the conclusion of Theorem~\ref{th:smithselmer} is in conformity with existing conjectures about variation of Selmer groups.  In particular, the theorem agrees with the heuristics formulated in \parencite{bklpr}[Conjecture 1.3], which posit that the $2^\infty$ Selmer group is distributed like the intersection $V \tensor (\Q_p/\Z_p) \cap W \tensor (\Q_p/\Z_p)$ where $V,W$ are random isotropic submodules in a large hyperbolic quadratic space over $\Z_p$.
\end{rema}

\begin{rema}
    The conditions on $A$ in Theorem~\ref{th:smithselmer} are not too serious.  For instance, Theorem~\ref{th:smithselmer} is proved in \parencite{smith1} for any elliptic curve over $\Q$ with no rational $2$-torsion or with all $2$-torsion rational and no rational $4$-isogeny.

    However, they are not mere artifacts of the proof: the variation of $\Sel^2$, for instance, really can be different in quadratic twists of elliptic curves with a single rational nontrivial $2$-torsion point~see e.g. \parencite{klagsbrunlemkeoliver}.  These issues are approached much more aggressively in \parencite{smith2}, where one sees that results like Theorem~\ref{th:smithselmer} can still be obtained if one restricts to a suitably chosen family of twists.  
    \label{rema:conditionsav}
\end{rema}

Thanks to the parity constraint on the kernel of an alternating matrix over $\F_2$, the Markov chain governing the distributions in Theorem~\ref{th:smithselmer} consists of two completely separate chains, one with absorbing state $0$ and the other with absorbing state $1$. The statement of Theorem~\ref{th:smithselmer} implies that the quadratic twists of $A$ are equidistributed between the two components; and indeed, it is known that the quadratic twists of an elliptic curve over $\Q$ have $\dim \Sel^2 E$ even half the time and odd half the time.  In the more general statement (\parencite{smith2}[Theorem 2.14]) of which Theorem~\ref{th:smithselmer} is a special case, one also needs to consider cases in which the quadratic twists of $A$ have constant parity; in such cases, the definition of $P^{\text{Alt}}(j|\infty)$ is modified to be supported on only a single parity of $j$.  This phenomenon -- that determining the parity of the dimension of a space carrying an alternating form is its own problem requiring separate methods from the main body of the work -- is rather general. The same kind of issue arises in the work already desribed on determining distributions from moments.  Consider, for instance, a random variable $X$ valued in finite-dimensional $\ell$-vector spaces obeying the Poonen-Rains distribution, which conjecturally governs the mod $\ell$ Selmer groups of random elliptic curves.  Write $X_\text{even}$ for $X$ conditioned on having even dimension, and $X_\text{odd}$ for $X$ conditioned on having odd dimension.  Then $X, X_\text{even}$, and $X_\text{odd}$ have all the same moments!  Indeed, this phenomenon  is already observed in the work of Heath-Brown, in the discussion after Theorem 2 of \parencite{heathbrown:congruent2}.  So any strategy which relies on computing moments and inferring the distribution will get stuck here, unless one has (as one fortunately often does) some separate means of computing the distribution of the parities.

\subsection{Applications}

Before we begin discussing the proofs of the main theorems, we mention some important applications.  First of all, Smith's theorem implies the long-studied {\em minimalist conjecture} that in the family of quadratic twists of a fixed elliptic curve over $\Q$, the proportion of twists $E^d$ with Mordell-Weil rank at least $2$ approaches $0$.  The argument is simple; because the Markov chain in Theorem~\ref{th:smithselmer} converges to one of the two absorbing states $0$ or $1$ with probability $1$, the probability that $r_{2^\infty}(E^d) \geq 2$ is $0$.  (To be precise, the results of \parencite{smith1,smith2} imply this for elliptic curves satisfying the technical conditions alluded to in Remark~\ref{rema:conditionsav}; but in a recent preprint~\parencite{smithbsd}, Smith has extended his proof of the minimalist conjecture to all elliptic curves over $\Q$.) But the $2^\infty$-Selmer rank is an upper bound for the Mordell-Weil rank, so we are done.  Indeed, it is known by ~\parencite[Thm 1.5]{monsky} that $0$ and $1$ occur equally often, so $r_{2^\infty}(E^d)$ takes the values $0$ and $1$ with probability $1/2$ each.  The results of Monsky also imply that the parity of the analytic rank of $E^d$ (the order of vanishing of the $L$-function $L(E_d,s)$ at the center of the critical strip) is the same as that of $r_{2^\infty}(E^d)$.  Thus, under the Birch-Swinnerton-Dyer conjecture that the analytic rank of $E$ is equal to the Mordell-Weil rank, we have that the analytic rank of $E^d$ is a nonnegative integer bounded by $r_{2^\infty}(E^d)$ and of the same parity as $r_{2^\infty}(E^d)$; it must thus be equal to $r_{2^\infty}(E^d)$ whenever that rank is $0$ or $1$, as it almost always is. One thus also obtains {\em Goldfeld's conjecture} that, in the limit, half the quadratic twists of $E$ have analytic rank $0$ and half have analytic rank $1$.

The critical advance here is not so much something special about the prime $2$, but rather that Smith's theorem can be applied to Selmer groups of arbitrarily large order.  For instance, Bhargava and Shankar proved in \parencite{bhargavashankar} that the average size of the $2$-Selmer rank of an elliptic curve (counted by height) is $3$.  A random variable with mean $3$ cannot be $2^2 = 4$ more than $75\%$ of the time, so one immediately gets an upper bound for the proportion of curves whose $2$-Selmer rank is at least $2$.  (One can do better still by considering parity.)  But no method of this kind can ever show that the proportion of curves with high rank is actually zero, unless our heuristics are badly wrong; for we believe that, for each $N$, there actually {\em is} a positive proportion of elliptic curves such that $\Sel^N(E)$ has high rank.  This proportion, however, is expected to get smaller as $N$ gets larger; so only once one gets control over $\Sel^N(E)$ for arbitrarily large $N$, as Smith does, does the minimalist conjecture become obtainable.

The more general statement of Smith's theorem~\parencite[Theorem 2.14]{smith2} applies to many classes of abelian varieties, including Jacobians of hyperelliptic curves.  A theorem of Coleman, using the method of Chabauty, gives an upper bound for the number of rational points on a curve whose Jacobian has rank smaller than its genus.  This fact dovetails nicely with theorems that prove that the rank of this Jacobian is rarely large as the curve moves in a family:  see \parencite{bhargavagross}, \parencite{poonenstoll} for notable prior uses of this strategy.  Smith, in Theorem 3.8 of \parencite{smith2}, proves the following theorem in this direction.  Let $X: y^2 = f(x)$ be a hyperelliptic curve over $\Q$ where $f(x)$ is a polynomial of degree $2g+1$ whose splitting field has Galois group $S_{2g+1}$.  Let $X_d$ be the quadratic twist with equation $dy^2 = f(x)$.  Then $X_d$ has at most $3$ rational points for $100\%$ of $d$.  In fact, the average value of $\max(0,|X_d(\Q)|-3)$ is zero.  The first statement essentially follows from Smith's main theorem followed by application of Chabauty (in a refined form from~\parencite{stoll:independence}); the second by the fact that the number of points on $X_d$ is bounded by an exponential function of the Mordell-Weil rank, which is bounded above by the same exponential of the $2^\infty$-Selmer rank, and the average of this function can be controlled by Smith's main theorem. 

Yet another application of the methods appears in \parencite{koymanssmith2cubes}, which proves that at least $31.95\%$ of integers are not the sum of two rational cubes.  (The expectation is that this proportion exists and is equal to $1/2$.)  This question has long been known to boil down to the distribution of Mordell-Weil rank of elliptic curves in a family of {\em cubic} twists $E_d: y^2 = x^3 - d$.  This is a good opportunity to recall that Smith's results are not restricted to the prime $2$; they allow you to study the $\ell$-primary Selmer groups of objects with an automorphism of order $\ell$ in a family of cyclic twists of order $\ell$.  In this case, they are applied to study the mod $3$ Selmer group of $E_d$ as $d$ varies.  Note that these curves, whence all their Tate modules, carry an automorphism of order $3$; though the fact that this automorphism is not defined over $\Q$ creates substantial difficulties, which is one reason this problem required work beyond the results of \parencite{smith2}.

These are some applications of the main theorems of Smith's papers; in section~\ref{s:koymanspagano} we will discuss more applications of the {\em techniques} of these papers in even broader contexts.

\subsection{Grids}

What makes it possible to make so much progress on $2$-primary Cohen-Lenstra when $p$-primary Cohen-Lenstra (over number fields, at least) has proven so difficult?  In a nutshell, it is because the $2$-primary class groups of different quadratic twists are related to each other, as we now explain.  The simplest manifestation of this phenomenon is already crucial in earlier work on mod $2$ Selmer in quadratic twist families of elliptic curves, which begins with the work of Heath-Brown on congruent number curves~\parencite{heathbrown:congruent1,heathbrown:congruent2} and has continued into the present \parencite{kmrmarkov,kane, park}.   Suppose $E/\Q$ is an elliptic curve and $E^d/\Q$ a quadratic twist.  Then $E[2]$ and $E^d[2]$ are isomorphic as Galois modules; so $\Sel^2(E)$ and $\Sel^2(E^d)$ are subgroups of the same Galois cohomology group $H^1(G_\Q, E[2])$, differing only insofar as the local conditions do.  This makes it substantially easier (albeit by no means straightforward) to understand how $\Sel^2$ varies in a family of quadratic twists; what you are really studying is how a family of local conditions varies as the quadratic twist changes.  The situation of $\Sel^p$ for odd primes $p$ is different; the groups $H^1(G_\Q, E[p])$ and $H^1(G_\Q, E^d[p])$ have different coefficients and there's no reason the two should be in any way related.

What about $\Sel^{2^k}$?  It is certainly not the case that $E[2^k]$ and $E^d[2^k]$ are isomorphic as Galois modules, since $-1$ no longer acts as the identity on these modules.  But the modules corresponding to different quadratic twists are not completely decoupled from one another, either.  Instead, there is a more subtle relation among the twists.  Let $d_1, d_2$ be two integers representing different classes in $(\Q^*)/ (\Q^*)^2$.  Let $N$ be an abelian group isomorphic to $(\Q_2/\Z_2)^r$ with a $G_\Q$-action $\rho_N: G_\Q \ra \Aut(N)$.  (We leave this Galois representation completely unrestricted to emphasize the generality of what we're about to say.)  If $d$ is a nonzero integer and $\chi_d: G_\Q \ra \pm 1$ the corresponding character of $G_\Q$, we denote by $N^d$ the Galois module isomorphic as $\Z_2$-module to $N$, and whose Galois action is the twisted one sending $\sigma$ to $\rho_N(\sigma) \chi_d(\sigma)$.

We may then consider the four Galois modules $N[4], N^{d_1}[4], N^{d_2}[4]$, and $N^{d_1 d_2}[4]$. These are pairwise non-isomorphic, but nonetheless they are related:

\begin{prop} $N^{d_1 d_2}[4]$ is a subquotient of $N[4] \oplus N^{d_1}[4] \oplus N^{d_2}[4]$ as $\Z/4\Z[G_\Q]$-modules.
\label{pr:grid}
\end{prop}

\begin{proof}
    Let $M$ be the quotient of $N[4] \oplus N^{d_1}[4] \oplus N^{d_2}[4]$ by the submodule spanned by
    $$
    2N[4] \oplus \set{0} \oplus 2N^{d_2}[4],\  
    2N[4] \oplus 2N^{d_1}[4] \oplus \set{0},\  
    \set{0} \oplus 2N^{d_1}[4] \oplus 2N^{d_2}[4].
    $$

    Write $\beta: N^{d_1 d_2}[4] \ra N[4]$ for the (non-Galois-equivariant) map which is an isomorphism of the underlying abelian groups and which has the property that, for every $\sigma \in G_\Q$, we have $\beta(\sigma n) = \chi_{d_1 d_2}(\sigma) \sigma \beta(n)$.  (That such an isomorphism exists is more or less the definition of the twist.)  Define $\beta_{d_1}$ and $\beta_{d_2}$ to be the corresponding maps from $N^{d_1 d_2}[4]$ to $N^{d_1}[4]$ and $N^{d_2}[4].$
    
    Now consider the map $\phi: N^{d_1 d_2}[4] \ra M$ obtained by sending an element $n$ to $\beta(n) \oplus \beta_{d_1}(n) \oplus \beta_{d_2}(n)$.

    We begin by showing that $\phi$ is Galois-equivariant.  For each $n \in N^{d_1 d_2}[4]$ and $\sigma \in G_{\Q}$, we have
    $$
    \phi(\sigma n) = \chi_{d_1 d_2}(\sigma) \beta(n) \oplus \chi_{d_2}(\sigma) \beta_{d_1}(n) \oplus \chi_{d_1}(\sigma) \beta_{d_2}(n)
    $$
    so $\phi(\sigma n) - \sigma \phi(n)$ is
    $$
    (\chi_{d_1 d_2}(\sigma)-1) \sigma \beta(n) \oplus (\chi_{d_2}(\sigma)-1) \sigma \beta_{d_1}(n) \oplus (\chi_{d_1}(\sigma)-1) \sigma \beta_{d_2}(n)
    $$
    We note that this element lies in $2N[4] \oplus 2N^{d_1}[4] \oplus 2N^{d_2}[4]$, and moreover is either $0$ (if $n \in 2N^{d_1 d_2}[4]$ or $\sigma$ is in the kernel of $\chi_{d_1} \oplus \chi_{d_2}$) or has exactly two out of three coordinates nonzero (otherwise.)  In either case, $\phi(\sigma n) - \sigma \phi(n)$ is $0$ in the quotient $M$, as claimed.

    We now show that $\phi$ is injective. Write $K_0$ for the subgroup of $(\Z/4\Z) \oplus (\Z/4\Z) \oplus (\Z/4\Z)$ spanned by $(2,2,0),(2,0,2),$ and $(0,2,2)$ and write $M_0$ for $(\Z/4\Z) \oplus (\Z/4\Z) \oplus (\Z/4\Z) / K_0$.  Then, as abelian groups, $M$ is $N[4]^{\oplus 3} / (N[4] \tensor_{\Z/4\Z} K_0)$ = $N[4] \tensor_{\Z/ 4\Z} M_0.$  The map $(\Z/4\Z) \ra M_0$ which sends $1$ to the image of $(1,1,1) \in M_0$ is visibly injective;  since $N[4]$ is free over $\Z/4\Z$, the same is true for $N[4] \ra N[4] \tensor M_0$.

    We have thus exhibited the Galois module $N^{d_1 d_2}[4]$ as a submodule of a quotient of $N[4] \oplus N^{d_1}[4] \oplus N^{d_2}[4]$, as claimed.
    
\end{proof}

(Proposition~\ref{pr:grid} appears as Example 7.16 in \parencite{smith1}, as a corollary of the more general Proposition 7.15.)

The utility of Proposition~\ref{pr:grid} is that, when $V$ is a twistable module, the $4$-Selmer group of $V^{d_1 d_2}$ is a subgroup of $H^1(G_F, V^{d_1 d_2}[4])$, which in turn is now revealed as a group we can describe in terms of the three other Galois cohomology groups $H^1(G_F, V[4]), H^1(G_F, V^{d_1}[4]),$ and $H^1(G_F, V^{d_2}[4])$.  This ties together the four $4$-Selmer groups corresponding to the four quadratic twists.

For higher powers of $2$, there is a similar but more complicated story. Note that the argument above does not exhibit $N^{d_1 d_2}[8]$ as a subquotient of $N[8] \oplus N^{d_1}[8] \oplus N^{d_2}[8]$; the problem is that $(4,0,0) = (2,2,0) + (2,-2,0)$, so the analogue of the module $M$ in $N[8] \oplus N^{d_1}[8] \oplus N^{d_2}[8]$ is killed by $4$, so $N^{d_1 d_2}[8]$ has no chance to inject into it. Rather, in order to study $2^k$-torsion, one has to consider a $k$-dimensional {\em grid}; that is, a set of $2^k$ quadratic characters spanned by a set of $k$ linearly independent elements in $(\Q^*)/(\Q^*)^2$.  Then one finds that, if $\chi$ is a character in the grid, $N^\chi[2^k]$ appears as a subquotient of the sum of the other $2^k-1$ twists of $N[2^k]$.  

This principle (or more accurately its generalization in \parencite{smith1}[Prop 7.15]) is central to Smith's method, since it supplies a method of finding relations among mod $2^k$ Selmer groups in a family of twists.   More specifically, it is well adapted for finding relations among twists which lie in a grid.

\begin{defi}
    Let $r$ be a positive integer and, for each $i$ in $\set{1,\ldots,r}$ let $X_i$ be a set of primes, all these sets being pairwise disjoint.  A {\em grid of twists} is a set of squarefree integers of the form $p_1 \cdot \ldots \cdot p_r$, where each $p_i$ is drawn from $X_i$.  We say $r$ is the {\em dimension} of the grid and the $|X_i|$ are the {\em sidelengths}.
\end{defi}

These relations between twists allow Smith to control the behavior of Selmer groups in a large grid of twists using knowledge about a smaller family of twists. 

\begin{rema}
As an analogy, imagine that you had a function $f$ on a group $G$ which had the property that $f(xy)$ was determined by the values of $f(x)$ and $f(y)$.  Then you would only need to know $f$ on a generating set of $G$ in order to determine $f$.  In the present context, what we have is a much fainter signal; for instance, one might need to know $2^k-1$ values in order to get information about one more!  But, at this level of abstraction, there is a long tradition of leveraging this kind of weak signal to get surprisingly strong control over a function.  Think of Dvir's theorem~\parencite{dvir} on the finite field Kakeya conjecture, which controls subsets of $\F_q^n$ (i.e. $0-1$ functions $f$ on $\F_q^n$) which contain no line.  That is to say: the mere fact that $f(x) = 1$ for $q-1$ points on a line tells you that $f(x) = 0$ for the remaining point.  The context of this kind which is closest to Smith's work, at least superficially, is that of inverse theorems for Gowers norms.  When $A$ is a finite abelian group, the $k$'th Gowers norm of a function $f: A \ra \C^\times$ is large when the product of $f$ over the vertices of many $k$-dimensional parallelepipeds in $A$ is close to $1$.  This can be thought of as a signal about one value of $f$ given $2^k-1$ others, and indeed~\parencite{taoziegler} one can show that the condition of having large Gowers norm imposes strong constraints on $f$.  (I cannot resist noting that these theorems are trivial for $k=1$, nontrivial but well-understood for $k=2$, and much more difficult and technical for $k \geq 3$, exactly in conformity with what happens when we study the variation of $2 \Cl_{\Q(\sqrt{-d})}[2^k]$.)
\end{rema}

\subsection{Equidistribution of pairings and symbols}

\label{s:equidistsymbols}

The main theorem of \parencite{smith1}, Theorem 4.18, says the distribution of the $2^k$-Selmer groups of the twists in $X$ does indeed approach the expected distribution as $X$ ranges over larger and larger grids satisfying some technical conditions.   (Here, ``larger and larger" combined with the technical conditions implies in particular that the primes in $X$, the dimension of $X$, and the sidelengths of $X$ are {\em all} growing).  We can only sketch the proof of Theorem 4.18 here. 

One key actor is the Cassels-Tate pairing on $2^{k-1} \Sel^{2^k}(V)$, which in the form needed for this paper is defined in \parencite{morgansmith}.  This pairing is automatically alternating if $V$ comes from an abelian variety, but does not carry this extra structure when $V$ is $\Z/N\Z$ (the case related to class groups.)  The kernels of this pairing compute $2^k \Sel^{2^{k+1}}(V)$~\parencite[Def 4.11]{smith1}. This is the core of the iterative strategy for proving Theorems~\ref{th:smithcg} and \ref{th:smithselmer}.  If one already knows that $\Sel^{2^k} V^d$ is distributed as one expects in a family of quadratic twists (say, in some large subset of a gigantic grid) then we can break this family up into subsets where $\Sel^{2^k} V^d$ is in an appropriate sense constant (at the very least, has constant dimension $N$) and try to show that the Cassels-Tate pairing is equidistributed among alternating $N \times N$ matrices over $\F_2$ in this family (or all $N \times N$ matrices, if the alternating structure is absent.)  That will imply that we have the desired distribution of $2^k \Sel^{2^{k+1}}(V)$ conditional on $2^{k-1} \Sel^{2^k}(V)$, which is exactly what we need.

This pairing, in turn, is determined by algebraic data related to the {\em symbols} of pairs of primes in $X$. The symbol of a pair of primes in $\bar{\Q}$, defined in \parencite[\S 3]{smith1} is one of the key algebraic aspects of the paper.  We emphasize that it does not depend merely on the two primes, but on an auxiliary extension $K/F$, which will depend on $V$, and which will need to be modified throughout the argument as we proceed through the iteration described above, 

When $K=F=\Q$, the symbol is simply the usual Legendre symbol of two primes.  So you can think of this passage from symbols to Cassels-Tate pairing to ``knowledge of the $2^{k+1}$-Selmer group from the $2^k$-Selmer group" as a vast generalization of the work of R\'{e}dei, which we briefly recall.  If $K$ is a quadratic field $K = \Q(\sqrt{N})$, where $N$ is a product of distinct primes $p_1 p_2 \ldots p_n$, the mod $2$ narrow class group of $K$ has dimension $n-1$, as was known to Gauss.  What about $\Cl_K[4]$?  Or, in the spirit of iterative construction, what about $2\Cl_K[4]$, which is the information required for determination of $\Cl_K[4]$ given $\Cl_K[2]$?  What R\'{e}dei shows is that the elementary abelian $2$-group $2\Cl_K[4]$ can be expressed in terms of the kernel of the $n \times n$ matrix $M_d$ over $\F_2$ whose $ij$ entry with $i \neq j$ is $1$ if and only the Legendre symbol $(\frac{p_i}{p_j}) = -1$, and whose diagonal entries are determined by the constraints that the row sums are 0.  In particular, this shows that if $d$ is drawn uniformly at random from a grid -- that is, if each $p_i$ is drawn uniformly from some large set $X_i$ of primes -- then each of the off-diagonal entries of the matrix $M_d$ is a Legendre symbol of a prime drawn at random from $X_i$ and a prime drawn independently at random from $X_j$.  One might well imagine that this matrix is well-modeled by a random $n \times n$ matrix (subject to the conditions imposed by quadratic reciprocity), at least as far as its rank goes; and indeed, this program is carried out successfully in \parencite{gerth:4rank,fouvrykluners} in order to control the distribution of $2\Cl_K[4]$ in families of quadratic fields.

In the setting considered by Smith, things become substantially more complicated.  Once $k$ is large, it is not exactly the case that the Cassels-Tate pairing on $2^{k-1} \Sel^{2^k}(V^d)$ is determined by the symbols between the primes dividing the twist parameter $d$.  We recall from the discussion above that, in any grid $X_0$ of dimension $k$ and sidelength $2$, there is some relation among the $2^k$ Galois modules $V^d$ where $d$ varies over the twists in $X_0$; and from this one can work out that there is some relation between the groups $2^{k-1} \Sel^{2^k}(V^d)$ as well.  What the symbols\footnote{or, more precisely, an arithmetic datum called the ``governing expansion" which is downstream of the symbols.} determine is {\em what relation this is}.  In particular, given the symbols, one obtains some family of relations that is satisfied among the Cassels-Tate pairings on $2^{k-1} \Sel^{2^k}(V^y)$ as $y$ ranges over some subgrid $Y$ of $X$.  (This is my paraphrase of \cite[Lemma 8.10]{smith1}.)  If such relations forced the Cassels-Tate pairing to be equidistributed as $d$ ranged over all of $Y$, we would be happy, but life is not quite that good.  Among the possible relations that might obtain, there are some which would {\em not} imply equidistribution, but these can be shown to be rare exceptions among all possible relations (This is my paraphrase of \cite[Proposition 7.9]{smith1}.)  A simple metaphor:  if you have a function $f$ on $\set{1,\ldots N}$ which satisfied the relation $f(n+1) = \lambda f(n)$ for some $\lambda$ on the unit circle, you would know that the average of $f$ was close to $0$, unless $\lambda$ was very close to $1$.  And if you didn't know that, but you did know you could break the set $\set{1,\ldots, N}$ into a union of almost-disjoint arithmetic progressions that almost covered the interval, and such that on each one of these arithmetic progressions with common difference $a$ you had a relation $f(n+a) = \lambda_i f(n)$, and if by a separate mechanism you knew that the $\lambda_i$ that provided the relations on those arithmetic progressions were well-distributed enough on the unit circle that you could choose the progressions in a way that very few of the $\lambda_i$ were near $1$, then, with some very careful bookkeeping, you would have a chance of showing that the average of $f$ on $\set{1, \ldots N}$ is close to $0$.  This is more than a metaphor (albeit less than a faithful description.)  In a sufficiently dense grid, one can show that the symbols of pairs of primes are reasonably well distributed (this is my paraphrase of \parencite[Theorem 5.2]{smith1}), and this allows us to evade those exceptional relations which would fail to force equidistribution of the Cassels-Tate pairing; Smith then breaks a large grid into subsets on each of which some relation between Selmer groups apply, and shows that the ``problematic relations" are rare enough that they can mostly be avoided, so that the Cassels-Tate pairing behaves itself on each of the subsets.

\begin{rema}
    In the paper of Koymans and Smith on sums of two rational cubes~\parencite{koymanssmith2cubes} mentioned earlier, the kind of equidistribution proved in \cite[Theorem 5.2]{smith1} is not enough; for that result, they need a {\em trilinear} equidistribution statement concerning symbols defined on {\em triples} of primes.  The simplest such symbols are called R\'{e}dei symbols; in some sense they are to Legendre symbols as Massey products are to cup products~\parencite{kim:triple}.
\end{rema}

Putting this together, one finds that the Cassels-Tate pairing is indeed equidistributed on all of $Y$, and showing that $X$ can be mostly covered by a union of subgrids to which this argument applies, we get that the Cassels-Tate pairing is equidistributed on $X$ (this is my paraphrase of \cite[Theorem 6.2]{smith1}).  More precisely, what Smith shows is that the Cassels-Tate pairing on $2^{k-1} \Sel^{2^k}(V^d)$ is equidistributed in a subfamily of $X$ called a ``higher grid class" over which $2^{k-1} \Sel^{2^k}(V^d)$ is constant; without such a restriction, the question doesn't really make sense, since one needs the dimension of $2^{k-1} \Sel^{2^k}(V^d)$ to be fixed in order to have a fixed space of pairings to equidistribute over.  

The main theorem \parencite[Theorem 4.18]{smith1} on the distribution of $\Sel^{2^k}$ follows, essentially by iterating this procedure; we start with a grid class  on which $\Sel^2$ is constant, show that the Cassels-Tate pairings equidistribute there, which provides the distribution of $2\Sel^4$ conditional on $\Sel^2$; this breaks our grid class into higher grid classes where $2\Sel^4$ is constant, and on each of these we apply the argument again to obtain the distribution of $4 \Sel^8$ conditional on $2 \Sel^4$, and so on.  At each stage, one needs to make sure the grids are large enough for the equidistribution statements on symbols to kick in; as one might imagine, this iterative process leads to convergence rates in the final result which are slow even by analytic number theory standards.  When the grids involve primes of size around $H$, the deviation of the distribution of $\Sel^{2^k}$ from its limiting value is bounded above by $\exp(-c (\log \log \log H)^{1/2})$.  It would be interesting to investigate, by experimental or theoretical means, what the actual rate of convergence might be.

\subsection{The fixed-point Selmer groups}

The careful reader will note that the above sketch is an induction without a base case.  Once we have restricted to a grid class on which $\Sel^2$ is constant, we can start to work on $\Sel^4$ -- but how do we control the variation of the first step, $\Sel^2$?  Or, more generally, when $V$ is a Galois module endowed with an automorphism of order $\ell$, how do we control the variation of $\Sel^\omega$, where $\omega = \zeta-1$?  Smith calls these the {\em fixed point Selmer groups}, since $V[\omega]$ is precisely the set of points fixed by $\zeta$.  Their study makes up the bulk of \parencite{smith2}. 
We can only gesture at the difficulties here.  For one thing, $\Sel^\omega$ very often does {\em not} approach a distribution as we vary over a family of cyclic $\ell$-twists, a phenomenon which we have already noted in the case of the $2$-Selmer groups of quadratic twists of elliptic curves.  Smith shows that for a well-chosen population of twists he calls {\em favored}, $\Sel^\omega$ does approach the expected distribution.  (This is my paraphrase of \parencite[Theorem 2.14]{smith2}.)

Even without restricting to the favored twists, one can show that the number of twists where $\dim \Sel^\omega$ is really problematically large is not too high.  (This is my paraphrase of \parencite[Theorem 2.6]{smith2}.)  This fact alone, combined with the fact that $\Sel^\omega(A)$ is an upper bound for the Mordell-Weil rank of an abelian variety $A$, is enough to show that the average rank of $A$ over a $\Z/\ell\Z$-extension of $\Q$ is finite. (This is my paraphrase of \parencite[Theorem 1.1]{smith2}.) Just as in \parencite{bhargavashankar}, one can get average rank results from control over a single mod $N$ Selmer group.  But the results of \parencite{smith2} do not yet allow one to control $\Sel^{\omega^k}$ for higher $k$ when $A$ is an abelian variety. So, as yet, we do not know in general that there is any fixed $N$ such that $\rank(A(K)) - \rank(A(F))$ is at most $N$ for $100\%$ of $\Z/\ell \Z$-extensions of $K/F$; unless $\ell = 2$, in which case this is Theorem~\ref{th:smithselmer} described above.  Why the difference between $2$ and odd $\ell$?  It comes down to the fact that the pairing on the Selmer groups of abelian varieties, which is alternating mod $2$, is instead {\em symmetric} modulo $\ell$.  Of course, when $1 = -1$ alternating and symmetric are not very different, except in one critical way: the diagonal entries of an alternating matrix over $\F_2$ are $0$, while the diagonal entries of a symmetric matrix over $\F_\ell$ can be whatever they please; and controlling these diagonal entries of the pairing is, for the moment, beyond the reach of the methods of \parencite{smith1,smith2}.

There is still one more step.  We have talked about how to get the appropriate probability distributions on large grids, but for most people it is more natural and desirable to compute averages over {\em intervals} -- one wants to know, for instance, a distribution over quadratic twists by $d$ where $d$ ranges over squarefree integers in the range $[0,H]$, not as $d$ ranges over the $1000$ products of one of the first ten primes, one of the next ten, and one of the third ten.  Section 8 of \parencite{smith2} explains how an interval can be approximated well enough by a union of grids of the kind treatable by Smith's methods that the probability distributions can be carried over from the grid context to the integral context.  Your correspondent, who is by no means an analytic number theorist, will say no more about this part.

One note:  for the study of $4$-ranks, it took quite a long time to get from average behavior with a fixed number of primes dividing the discriminant (Gerth) to the desired average over all discriminants (Fouvry-Kl\"{u}ners), and this created a sense that these problems become easier if you restrict the number of prime factors of the discriminant.  The situation is now quite the opposite!  Smith's arguments rely critically on being able to work with very large and flexibly chosen sets of ramified primes.  Proving that, say, the twists of an elliptic curve by a quadratic field of {\em prime} discriminant are $100 \%$ of rank $0$ or $1$ seems out of reach for Smith's methods in their current form.  As for class groups of quadratic field with almost-prime discriminant, partial results are known up to the $16$-rank~\parencite{koymansmilovic}, but beyond that, as one expert in the field told me, ``all is darkness."

\section{The work of Koymans and Pagano on the negative Pell equation}

\label{s:koymanspagano}

The Pell equation $x^2 - d y^2 = 1$, where $d$ is a fixed squarefree positive integer and $x$ and $y \neq 0$ are integers to be solved for, is one of the oldest and most widely studied Diophantine equations; so old that Archimedes thought about it in some form.  Brahmagupta made major progress on understanding its solutions in the 7th century CE.  One mathematician whose contribution to the Pell equation is more questionable, ironically, is John Pell (1611-1685), though he certainly worked in number theory and the equation was included in a textbook to which he contributed greatly.\footnote{Pell told the biographer John Aubrey that ``when he solves a Question, he straines every nerve about him, and that now in his old age it brings him to a Loosenesse."  Very relatable even for modern number theorists, especially when contemplating what must have been the nerve-straining difficulty of building the papers discussed here today!}

The Pell equation always has solutions; in modern terms, we would say that the real quadratic field $\Q(\sqrt{d})$ always has a unit $x_0 + y_0 \sqrt{d}$ other than $\pm 1$; the norm of this unit is a unit in $\Z$, which is to say $\pm 1$; so its square is a unit $x + y \sqrt{d}$ whose norm is $1$, which is to say that $x^2 - d y^2 = 1$.  Indeed, any even power of the unit will do, which means that once you have one solution you have infinitely many; this is what Brahmagupta worked out a millennium and a half ago.

But let's go back to that unit $x_0 + y_0 \sqrt{d}$, which yields a solution to $x_0^2 - dy_0^2 = \pm 1$.  It is certainly possible for the $-1$ case to be achieved; for instance, when $d=5$ we have $2^2 - 5 \cdot 1^2 = -1.$  But it is not {\em always} achieved.  For one thing, the negative Pell equation
\begin{equation}
    x^2 - d y^2 = -1
    \label{eq:negpell}
\end{equation} might not even have a solution over $\Q$.  It is not hard to show that \eqref{eq:negpell} has a solution over $\Q$ if and only if every odd prime divisor of $d$ is congruent to $1$ modulo $4$.  When this is the case, we say that $d$ satisfies the {\em Pellian condition.}  But the Pellian condition does not guarantee that \eqref{eq:negpell} has a {\em integer} solution, which is far more subtle. For example, \eqref{eq:negpell} has no integer solution with $d=34$, despite the existence of a rational solution $(5/3)^2 - 34 \cdot (1/3)^2 = -1$.

Whatever determines the solvability of \eqref{eq:negpell}, it cannot be something as simple as a condition on the primes dividing $d$, since Dirichlet showed that \eqref{eq:negpell} is solvable with $d=p$ for {\em every} prime $p$ congruent to $1$ mod $4$.  The condition must have something to do with interactions among the prime factors of $d$. And indeed, Dirichlet also knew that $d = pq$ yields a solvable \eqref{eq:negpell} whenever the Legendre symbol $(\frac{p}{q})$ is $-1$. Dirichlet also gave a sufficient condition for insolubility of \eqref{eq:negpell} in terms of biquadratic symbols; together, these results showed that (with $c,c',c'' > 0$) of the $c X / \sqrt{\log X}$ values of $d$ in $[0,X]$ satisfying the Pellian condition, at least $c' X / \log X$ have \eqref{eq:negpell} solvable and at least $c'' X \log \log X / \log X$ have \eqref{eq:negpell} unsolvable.  Nagell conjectured in 1932 that both these quantities in fact accounted for a positive proportion of the discriminants.  And \parencite{stevenhagen} sharpens this to a precise conjecture for the constant.  It is this conjecture of Stevenhagen that Koymans and Pagano have proved.  Non-solvable instances of \eqref{eq:negpell} are actually fairly rare among small discriminants, occurring for only $14\%$ of the discriminants below $10,000$~\parencite{stevenhagen}[Table 1].  In fact, \eqref{eq:negpell} is insoluble almost half the time.

\begin{theo}[Theorem 1.1, \parencite{koymanspagano}]
    Stevenhagen's conjecture is true.  In other words:  the proportion of positive squarefree $d < X$ with all odd prime divisors congruent to $1$ mod $4$ such that $x^2 - d y^2 = -1$ has a solution in integers $x,y$ approaches
    \begin{equation}
    1 - \prod_{j \geq 1, j\  \text{odd}} (1-2^{-j}) = 0.58057\ldots
    \end{equation}
as $X \ra \infty.$ 
\end{theo}

The reader may now ask:  what does any of this have to do with the Cohen-Lenstra conjectures?  The key is that the negative Pell equation can be thought of as a question about the class group.  Write $K = \Q(\sqrt{d})$.  The narrow class group of $K$, which we shall denote $C(K)$\footnote{With apologies to readers of \parencite{koymanspagano}, where the narrow class group is denoted $\Cl(K)$; but in this writeup that notation is reserved for the ordinary class group.}, is the quotient of the group of fractional ideal classes by the group of principal ideals generated by totally positive elements of $\OO_K$. 

\begin{prop}
    The negative Pell equation $x^2 - d y^2 = -1$ is solvable if and only if the principal ideal $(\sqrt{d})$ is trivial in the narrow class group $C(\Q(\sqrt{d}))$.
    \label{pr:negpellcg}
\end{prop}

\begin{proof}
If $(\sqrt{d})$ is trivial in the narrow class group, it is generated by a totally positive element $a$, which must be $\sqrt{d}u$ for some unit $u \in \OO_K^*$.  Since the embeddings of $\sqrt{d}$ have opposite signs, so do the embeddings of $u$, which means $u$ has norm $-1$.  Conversely, if there is a unit $u$ of norm $-1$, then $\sqrt{d} u$ is the desired totally positive generator of $(\sqrt{d})$.
\end{proof}

Another equivalent formulation is that \eqref{eq:negpell} is solvable if and only if the narrow class group and the usual class group are equal.

From Proposition~\ref{pr:negpellcg}, we see that statistical questions about the solvability of the negative Pell equation are equivalent to statistical questions about the class group.  But more: because $(\sqrt{d})$ is clearly an element of order $2$ in $C(K)$, these are questions that require only knowledge of the $2$-primary part of $C(K)$.  So we are squarely in Cohen-Lenstra territory.  Note that this point of view nicely explains Dirichlet's results.  When $p$ is a prime congruent to $1$ mod $4$, we know by genus theory that $|C(K)|$ is odd, so the order $2$ element $(\sqrt{d})$ is certainly $0$.  When $d = pq$, genus theory tells us that $C(K)/2C(K) \cong \Z/2\Z$, and also that $(\sqrt{d})$ is trivial in $C(K)/2C(K)$. So if $(\sqrt{d})$ is an element of $2 C(K)$ of order $2$ and is {\em not} trivial, it must be the case that $C(K)$ has an element of exact order $4$; the criterion for this to be the case is exactly that $(\frac{p}{q}) = 1$.  Thus, when $(\frac{p}{q}) = -1$, we know that $(\sqrt{d})$ is trivial, which is to say that \eqref{eq:negpell} is solvable.

This story, together with what we've seen in Smith's work, already suggests what in broad outline the argument of \parencite{koymanspagano} may look like.  We will want to know the probability that $(\sqrt{d})$ vanishes in $C(K)/2C(K)$; if it does, we will want to know the probability that this element of $2C(K)$ vanishes in $2C(K)/4C(K)$; if it does, we will take a deep breath and look at $4C(K)/8C(K)$, and so on.  If everything goes right, we will expect to end up with a Markov model which not only captures the behavior of $C(K)[2^\infty]$ as in \parencite{smith1}, but of $C(K)[2^\infty]$ together with a specified $2$-torsion element in that group.

Let's imagine what such a Markov model might look like.  We have already described the Markov model that describes the $2$-primary part of the class group of a random quadratic imaginary field.  It turns out that the narrow class groups of random real quadratic fields are governed by a similar Markov model; but in this case, the transition probability from $r_{2^k} = n$ to $r_{2^{k+1}} = j$ is given by the probability that the kernel of a random $n \times (n+1)$ matrix over $\F_2$ is $j$, rather than being governed by the kernel of a square $n \times n$ matrix as in Theorem \ref{th:smithcg}.  This ``slightly rectangular'' Markov process is in keeping with what happens in the classical Cohen-Lenstra setting, where the class group of a quadratic imaginary field behaves like a random abelian group, while the class group of a real quadratic field behaves like a random abelian group modulo a random element. Adding the $n+1$'st row to the square matrix imposes one more random relation, which is tantamount to modding out the kernel of the square matrix by a random element of that kernel.

For the negative Pell problem, there are two features that complicate this model.  First of all, there is a difference at the very first layer (or second, depending on how you count).  The dimension of $C(K)/2C(K)$ is $t-1$, where $t$ is the number of ramified primes in $K$.  The dimension of $2C(K)[4]$ is $t-1-r$, where $r$ is the rank of the R\'{e}dei matrix we met in \S\ref{s:equidistsymbols}, whose off-diagonal entries are Legendre symbols of the primes ramified in $K$.  When all the odd ramified primes are congruent to $1$ modulo $4$, which is our ongoing hypothesis, this matrix is {\em symmetric}, by quadratic reciprocity.  So the $4$-rank in this Pellian setting is modeled by the kernel of a random large {\em symmetric} matrix over $\F_2$, not the kernel of an arbitrary large random matrix as in Theorem~\ref{th:smithcg}.

The careful reader should now protest: I just made a big deal about the matrix being slightly rectangular, not square, and here it is being square again!  This brings us to the second point.  We need to keep in mind that in the real quadratic setting, $C(K)$ always has a canonical element in it; namely, the class of $(\sqrt{d})$.  Suppose we have already established that $2^{k-1} C(K)[2^k]$ has rank $n$ and that $(\sqrt{d})$ lies in $2^{k-1}C(K)$.  Then $2^k C(K)[2^{k+1}]$ will be expressed as the quotient of $2^{k-1} C(K)[2^k]$ by $n$ random elements, and in addition by the class of $(\sqrt{d})$.  Now if $(\sqrt{d})$ happens to be a multiple of $2^k$, it pairs to zero with everything in $2^{k-1} C(K)[2^k]$; this makes the ''extra" column zero, which means that $2^k C(K)[2^{k+1}]$ in this case will be isomorphic to the kernel of a random $n \times n$ matrix over $\F_2$.

If this account is correct, and if every matrix mentioned above is taken to be random, the negative Pell problem is governed by a Markov chain with a state labeled NO and a state labelled $n$-YES for each nonnegative integer $n$.  At time $k$, being in the state NO means that $(\sqrt{d})$ does not lie in $2^{k-1}C(K)$.  Being in the state $n$-YES means that $(\sqrt{d})$ {\em does} lie in $2^{k-1}C(K)$, and that the dimension of $2^{k-1}C(K)[2^k]$ is $n$.  This yields the following expected transition probabilities.  At time $2$, we are in state $n$-YES if a random large symmetric matrix over $\F_2$ has kernel of dimension $n$, and are never in state NO.  (This is because the Pellian condition already implies that $(\sqrt{d})$ lies in $2C(K)$.)  Now, at time step $k$, if we are in $n$-YES, the probability of moving to NO at time $k+1$ is the probability that $(\sqrt{d})$ does not lie in $2^k C(K)$. Again assuming everything is as random as it can be, this probability is $1-|2^{k-1}C(K)/2^kC(K)|^{-1} = 1-2^{-n}$.  Assuming on the other hand that $(\sqrt{d})$ {\em does} lie in $2^k C(K)$, the last column of the $n \times (n+1)$ matrix determining $2^kC(K)[2^{k+1}]$ is zero, so $2^kC(K)[2^{k+1}]$ has dimension $j$ with the probability we've denoted $P^{\text{Mat}}(j|n)$.  Thus, the overall transition probability from $n$-YES to $j$-YES is $2^{-n} P^{\text{Mat}}(j|n)$.

What Koymans and Pagano prove is that all these heuristics are correct, albeit with an even more hilariously slow rate of convergence than we saw in Smith's theorems above.  Here is the precise statement.  Following the authors, we denote by $\mathcal{D}(X)$ the set of Pellian discriminants $d$ in $[0,X]$, and by $\mathcal{D}_{i,n_i}(X)$ the subset consisting of all those Pellian discriminants $d$ such that the $2^i$-rank of $C(\Q(\sqrt{d}))$ is $n_i$ and $(\sqrt{d})$ lies in $2^{i-1} C(\Q(\sqrt{d}))$.

\begin{theo}[\parencite{koymanspagano}, Theorem 8.1]
There are real numbers $A, c, X_0 > 0$ such that for all reals $X > X_0$, all integers $m \ge 2$ and all sequences of integers $n_2, \dots, n_{m+1} \ge 0$,

\[
\left| \left| \bigcap_{i=2}^{m+1} \mathcal{D}_{i,n_i}(X) \right| - \frac{P^{\text{Mat}}(n_{m+1}|n_m)}{2^{n_m}} \cdot \left| \bigcap_{i=2}^m \mathcal{D}_{i,n_i}(X) \right| \right| \le \frac{A \cdot |\mathcal{D}(X)|}{(\log \log \log \log X)^{\frac{c}{m^2 6^m}}}.
\]
\label{th:kpmarkov}
\end{theo}

The Markov chain we described above has two absorbing states, NO and 0-YES.  The probability that a Pellian discriminant $d$ admits a solution to the negative Pell equation is then the probability that the chain eventually reaches 0-YES.  This comes out to be the probability conjectured by Stevenhagen, and now proven by Koymans and Pagano.

One notes that, conditional on terminating at 0-YES, the distribution on $2C(K)[2^\infty]$ is very similar to that of quadratic imaginary fields; the $2C(K)[4]$ part is different, thanks to the symmetry of the R\'{e}dei matrix, but the Markov chain from that point onward is identical with that in Smith's result reproduced here as Theorem~\ref{th:smithcg}.  One might think of the real quadratic fields with negative Pell solutions as being a family of ``fake quadratic imaginary fields" -- as when $d$ is negative, the narrow class group is equal to the class group, and the Markov chain is the one worked out by Smith. 

The method by which Koymans and Pagano prove Theorem~\ref{th:kpmarkov} is certainly indebted to Smith's work, but goes further.  One key difference (but far from the only one) is that, in Smith, it is crucial to have the flexibility to choose grids so that the symbols you have to compute are essentially pairings of random primes chosen independently.  In the negative Pell setting, some of the symbols involved, no matter what grid you work with, have an argument that's forced upon you; the class of $(\sqrt{d})$.  The difficulty here is eventually overcome by the use of a novel reciprocity law introduced in the earlier paper~\parencite{koymanspagano:higherredei}.  The subject remains in a state of rapid development, and it seems likely that time will distill the new spray of reciprocity laws and equidistribution statements found in \parencite{koymanspagano:highergenus,koymanssmith2cubes}, etc. into a more general theory whose shape for now remains only partially in view.

\section{Into the nonabelian}

\label{s:nonabelian}

We have actually undersold the results of \parencite{lwzb} and \parencite{landesmanlevy} discussed in \S\ref{s:clm}.  By class field theory, the class group of $K$ is isomorphic to the Galois group of the maximal everywhere-unramified abelian extension of $K$.  If $K^{nr}$ is the compositum of all unramified algebraic extensions of $K$, this group can be written as the abelianization of $\Gal(K^{nr}/K)$, and we ask how this group varies as $K$ moves in a family of number fields or function fields.

But who said we need to abelianize?  Why not ask about the variation of the whole profinite group $\Gal(K^{nr}/K)$ as $K$ varies through a family of number fields?

Many difficulties present themselves at once.  The isomorphism classes of finite abelian groups are countable in number and we know how to describe them all.  But $\Gal(K^{nr}/K)$ need not be finite, as we know from the Golod-Shafarevich construction~\parencite{golodshafarevich}.  We do not even know whether this group is topologically finitely generated. We may make our life easier by choosing an odd prime $p$ and considering only the maximal pro-$p$ quotient $\Gal(K^{nr;p}/K)$ of $\Gal(K^{nr}/K)$, which is closer to the spirit of Cohen-Lenstra.  Then at least our groups are topologically finitely generated (this being the case for any pro-$p$ group with finite abelianization) but they still may be infinite by Golod-Shafarevich.  The idea that there is nonetheless a meaningful conjecture to make was championed by Nigel Boston, who died in 2024 and whose early contribution to this line of inquiry has been essential to its development.  When $K$ is a quadratic imaginary field, $\Gal(K^{nr;p}/K)$ carries an involutive automorphism $\sigma$ coming from the action of $\Gal(K/\Q)$, which acts as $-1$ on $\Gal(K^{nr;p}/K)^{ab}$.    The paper \parencite{bbh} defines a measure $\mu_{BBH}$ on the set of isomorphism classes of finitely generated pro-$p$ groups with finite abelianization and endowed with such an automorphism, and proposes the heuristic that for any ``reasonable" function $f$ on the isomorphism classes of such pro-$p$ groups, the average of $f(\Gal(K^{nr;p}/K))$ as $K$ ranges over imaginary quadratic fields is the integral of $f$ against $\mu_{BBH}$.  As one would hope, the pushforward of $\mu_{BBH}$ from finitely generated pro-$p$ groups to finite abelian groups is just the Cohen-Lenstra measure.

As in the Cohen-Lenstra setting, it is reasonable to ask about moments.  In other words:  if $H$ is a finite $p$-group with a generator-inverting automorphism $\sigma$, we can define the $H$-moment of a measure $\mu$ to be the expected number of $\sigma$-invariant surjections from $G$ to $H$, where $G$ is drawn from $\mu$.  \parencite{bostonwood} proves the remarkable theorem that every moment of the Boston-Bush-Hajir distribution $\mu_{BBH}$ is $1$!

So a reasonable way of investigating the non-abelian Cohen-Lenstra conjecture of Boston, Bush, and Hajir is to try to compute these moments when $G$ is $\Gal(K^{nr;p}/K)$ with $K$ a random quadratic imaginary field.  More precisely, $K$ is a random quadratic imaginary field with discriminant in $[-X,0]$, and we hope that the $H$-moment $m_{H,X}$ of $\Gal(K^{nr;p}/K)$ always approaches $1$ as $X \ra \infty$.

The objects counted by $m_{H,X}$ are everywhere unramified $H$-extensions of $K$ with a specified action of $\Gal(K/\Q)$ on $H$.  We can also think of such an extension as a $H \rtimes (\Z/2\Z)$-extension of $\Q$, ramified only at primes dividing the discriminant $D_K$ and with inertia groups of order $2$ at those primes.  This connection is older than the Cohen-Lenstra conjecture itself; \parencite{davenportheilbronn}[Theorem 3] computes the average size of $\Cl_K[3]$ as $K$ ranges over quadratic imaginary fields by relating this quantity to the number of $S_3$-extensions of $\Q$ with discriminant at most $X$ whose inertia group at every odd prime has order $2$.  (Note that the average size of $\Cl_K$ is the $(\Z/3\Z)$ moment, and $(\Z/3\Z) \rtimes (\Z/2\Z)$ with $\Z/2\Z$ acting as $-1$ is a fancy name for $S_3$.)

This brings us into contact with {\em Malle's conjecture}, another very active current in arithmetic statistics which could support a separate survey paper of its own.  Malle's conjecture concerns questions of the form:  how many $H$-extensions of $\Q$ are there with discriminant of absolute value at most $X$?  (And more generally, one can ask about extensions of arbitrary number fields or function fields, extensions with restrictions on ramification type, extensions counted by invariants other than discriminant, etc.)  In the case considered in \parencite{bostonwood}, the moments being $1$ for every $H$ indeed conforms with what Malle predicts.

The paper \parencite{lwzb} takes this approach even further.  Here, one considers an arbitrary finite group $\Gamma$, one lets $K$ range over $\Gamma$-extensions of $\Q$ or of $\F_q(t)$ whose ramified primes have product of absolute value at most $X$, and asks about $\Gal(K^{nr}/K)$; not just the pro-$p$ quotient, but the whole profinite shebang.   (Well, almost the whole profinite shebang; we will still restrict to the maximal prime-to-$2|\Gamma|$ quotient which you'll note still keeps us separated from the $2$-group extensions of quadratic extensions considered by Smith and Koymans-Pagano.) This Galois group is a profinite group carrying an action of $\Gamma$. Liu, Wood, and Zureick-Brown define a probability distribution on such groups.  Roughly, it is defined as follows.  Let $F_{n,\Gamma}$ be the {\em free profinite $\Gamma$-group on $n$ generators}; it is generated by $n|\Gamma|$ elements $x_{i,\gamma}$, and is supplied with a $\Gamma$-action by the rule $\gamma' \cdot x_{i,\gamma} = x_{i, \gamma' \gamma}$.  We then let $\FF_{n, \Gamma}$ be the closed subgroup of $F_{n,\Gamma}$ topologically generated by the elements $x_{i,id} x_{i,\gamma}^{-1}$ and their images under $\Gamma$. 

Having done this, they define $X_{n, \Gamma}$ to be a random variable valued in profinite groups defined as the quotient of $\FF_{n,\Gamma}$ by the $(n+1)|\Gamma|$ relations $r^{-1} \gamma(r)$, as $\gamma$ ranges over $\Gamma$ and $r$ ranges over a set of $n+1$ elements chosen Haar-uniformly from $\FF_{n,\Gamma}$.  Finally, they show that as $n \ra \infty$, the distributions of these random variables converge to a measure on profinite groups with $\Gamma$-action, which they call $\mu_\Gamma$.  It is this measure which they conjecture governs the variation of the prime-to-$2|\Gamma|$ part of $\Gal(K^{nr}/K)$ as $K$ varies over $\Gamma$-extensions~\parencite[Conjecture 1.3]{lwzb} You might think of this as a non-abelian Cohen-Lenstra-Martinet conjecture ``at many odd primes at once."  When $\Gamma = \Z/2\Z$, the pushforward of their measure to the pro-$p$ quotient is $\mu_{BBH}$, and thus the pushforward to the abelianization of the pro-$p$ quotient is Cohen-Lenstra measure.

This measure has some properties which one might not have expected in advance.  For example, it assigns probability $0$ to any group which doesn't satisfy a certain lifting property the authors call ``property E."  Happily, the unramified Galois groups of $K$ automatically satisfy this property.

The history of the Cohen-Lenstra-Martinet conjectures, which, as you'll recall, had to be repeatedly modified as new phenomena were discovered, may reasonably give one pause here.  How do you know there's not some other subtle property which makes a profinite group $G$ unable, or even differentially unlikely, to appear as an unramified Galois group, and which is not accounted for in the conjecture in its present form?  Here we come to one of the most appealing features of \parencite{lwzb}, and of its successor papers as well.  The non-abelian Cohen-Lenstra-Martinet conjecture over $\F_q(t)$ can be expressed, just as in \ref{s:ffmoments}, as a question about counting points on certain Hurwitz spaces over finite fields.  The results of \parencite{wood:liftinginvariant} again provide a complete description of the connected components of these spaces, their dimensions, and the Frobenius action on them.  Deligne's bounds on Frobenius eigenvalues thus allow us to estimate the number of points on any such space as $q \ra \infty$.  The upshot is this: if $X_{H,q,n}$ is the number of unramified $H$-extensions of $K$, where $K$ is drawn at random from the set of $\Gamma$-extensions of $\F_q(t)$ the product of whose ramified primes has norm $q^n$, then the mean of $X_{H,q,n}$ approaches a limit as $q \ra \infty$, and the limit of {\em this} limit as $n \ra \infty$ is equal on the nose to the $H$-moment of the conjectured distribution $\mu_\Gamma$.  This is very powerful evidence that the choice of $\mu_\Gamma$ is correct, given that we are in the regime where moments determine distributions.\footnote{That moments determine distributions for this problem wasn't known at the time \parencite{lwzb} was written; that came later, with the work of Sawin and Wood already discussed.}  

The results of \parencite{landesmanlevy} provide even more compelling evidence; their theorems apply to the Malle conjecture generally, not just those cases related to Cohen-Lenstra, and their results allow one to show in many cases that the mean of $X_{H,q,n}$ approaches a limit as $n \ra \infty$ with {\em fixed} $q$.  However, it remains a boundary for now that this is known only when $q$ is sufficiently large relative to $H$; in other words, for any given $\F_q(t)$, there are only finitely many moments of the Liu-Wood-Zureick-Brown conjecture which can currently be shown to be correct.

As we have already seen, the presence of roots of unity in the base field complicates matters.  Forthcoming work of Sawin and Wood~\parencite{sawinwoodav} promises to extend the conjectures to the case of arbitrary base field.  This new work has an interesting feature which distinguishes it from everything that came before.  As we've seen, the distribution $\mu_\Gamma$ is determined by a certain {\em construction} of a random group, in the form of ``random generators and relations" -- this is a de-abelianized version of the original Friedman-Washington conception of the Cohen-Lenstra distribution as the distribution on the cokernel of a random large $\ell$-adic matrix.  One then computes the moments of the groups constructed thereby and compares them with moments arrived at by some other method.  \parencite{sawinwoodav} goes in the opposite direction.  What the moments should be is worked out by an algebro-geometric computation in the function field case, just as in \parencite{lwzb}.  But now the distribution is reconstructed directly from the moments using the technique introduced in \parencite{sawinwoodmoments}.  It is, in some sense, ``a distribution without a random variable."  It is an interesting question whether there's a more traditional construction of a random profinite group that obeys the distribution in  \parencite{sawinwoodav}.

The results described above are all concerned with prime-to-$|\Gamma|$ extensions of $\Gamma$-extensions of some base global field $F$.  What if we relax this assumption, or even enforce its opposite?  An example of such a question would be the study of everywhere-unramified $2$-extensions of random quadratic extensions of $F$.  If we append "abelian" to "everywhere-unramified," we are asking about $2$-primary class groups of quadratic extensions, the very question appearing in work of Smith, Koymans, and Pagano already described.  The study of $p$-primary class groups of $\Gamma$-extensions, where $p$ is now allowed to divide $|\Gamma|$, is developing rapidly; see e.g. \parencite{liu:cgae} and \parencite{koymansliu} for the current state of affairs.  But the maximal unramified $p$-extension of a random $\Gamma$-extension is, I believe, still not well-understood in this case.

One can imagine going still further.  When $\Gamma$ is a $p$-group, an unramified pro-$p$ extension of a $\Gamma$-extension $K/\Q$ is a pro-$p$ extension of $\Q$ whose ramified primes are only those of $K$. So one might ask the following question: let $N$ be a random integer and let $G_{N,p}$ be the Galois group of the maximal pro-$p$ extension of $\Q$ unramified away from $N$.   (For simplicity, let's require $(N,p) = 1$.) The group $G_{N,p}$ will not have a distribution, because the rank of its abelianization is essentially the number of primes dividing $N$ and congruent to $1$ mod $p$, which has infinite average as $N$ grows; this is the same reason that $\Cl_K[2]$ doesn't have a distribution as $K$ varies over quadratic fields.  Still, one may wish to study its statistics.  One way to address this is by fixing the number of primes dividing $N$; this is the approach taken in \parencite{bostonellenberg}, which formulates a conjecture for the probability that $G_{N,p}$ is isomorphic to a specified finite $p$-group when $N$ is the product of two random primes congruent to $1$ mod $p$.  Another approach, more in the spirit of Buell, Gerth, Smith, etc., would be to ask: what are the statistics of $G_{N,p}$ which play the role that the ranks of $2^m\Cl_K[2^{m+1}]$ do for $\Cl_K[2^\infty]$?  In other words, which statistics of $G_{N,p}$ {\em do} approach a distribution as $N$ varies over a large range of integers?

\printbibliography

@incollection{cohenlenstra,
  title={Heuristics on class groups of number fields},
  author={Cohen, Henri and Lenstra Jr, Hendrik W.},
  booktitle={Number Theory Noordwijkerhout 1983: Proceedings of the Journ{\'e}es Arithm{\'e}tiques held at Noordwijkerhout, The Netherlands July 11--15, 1983},
  pages={33--62},
  year={1983},
  publisher={Springer}
}

@inproceedings{wood:moments,
  title={Probability theory for random groups arising in number theory},
  author={Wood, Melanie Matchett},
  booktitle={Proc. Int. Cong. Math},
  volume={6},
  pages={4476--4508},
  year={2022}
}

@article{buell,
  title={Class groups of quadratic fields},
  author={Buell, Duncan A.},
  journal={Mathematics of Computation},
  volume={30},
  number={135},
  pages={610--623},
  year={1976}
}

@article{malleroots,
  title={Cohen--Lenstra heuristic and roots of unity},
  author={Malle, Gunter},
  journal={Journal of Number Theory},
  volume={128},
  number={10},
  pages={2823--2835},
  year={2008},
  publisher={Elsevier}
}

@article{wangwood,
  title={Moments and interpretations of the Cohen--Lenstra--Martinet heuristics},
  author={Wang, Weitong and Wood, Melanie Matchett},
  journal={Commentarii Mathematici Helvetici},
  volume={96},
  number={2},
  pages={339--387},
  year={2021}
}

@article{lwzb,
  title={A predicted distribution for Galois groups of maximal unramified extensions},
  author={Liu, Yuan and Wood, Melanie Matchett and Zureick-Brown, David},
  journal={Inventiones mathematicae},
  volume={237},
  number={1},
  pages={49--116},
  year={2024},
  publisher={Springer}
}

@article{wood2010,
  title={On the probabilities of local behaviors in abelian field extensions},
  author={Wood, Melanie Matchett},
  journal={Compositio Mathematica},
  volume={146},
  number={1},
  pages={102--128},
  year={2010},
  publisher={London Mathematical Society}
}

@article{sawinwood3manifolds,
  title={Finite quotients of 3-manifold groups},
  author={Sawin, Will and Wood, Melanie Matchett},
  journal={Inventiones mathematicae},
  volume={237},
  number={1},
  pages={349--440},
  year={2024},
  publisher={Springer}
}

@article{sawinwoodmoments,
  title={The moment problem for random objects in a category},
  author={Sawin, Will and Wood, Melanie Matchett},
  journal={arXiv preprint arXiv:2210.06279},
  year={2022}
}

@article{smith1,
  title={The distribution of $\ell^\infty$-Selmer groups in degree $\ell$ twist families I},
  author={Smith, Alexander},
  journal={Journal of the American Mathematical Society},
  volume={39},
  number={1},
  pages={1--72},
  year={2026}
}

@article{smith2,
  title={The distribution of $\ell^\infty$-Selmer groups in degree $\ell$ twist families II},
  author={Smith, Alexander},
  journal={Journal of the American Mathematical Society},
  volume={39},
  number={2},
  pages={453--514},
  year={2026}
}

@article{koymanssmith2cubes,
  title={Sums of rational cubes and the $3 $-Selmer group},
  author={Koymans, Peter and Smith, Alexander},
  journal={arXiv preprint arXiv:2405.09311},
  year={2024}
}

@article{klagsbrunlemkeoliver,
  title={The distribution of 2-Selmer ranks of quadratic twists of elliptic curves with partial two-torsion},
  author={Klagsbrun, Zev and Lemke Oliver, Robert J},
  journal={Mathematika},
  volume={62},
  number={1},
  pages={67--78},
  year={2016},
  publisher={Wiley Online Library}
}

@article{smithbsd,
  title={The Birch and Swinnerton-Dyer conjecture implies Goldfeld's conjecture},
  author={Smith, Alexander},
  journal={arXiv preprint arXiv:2503.17619},
  year={2025}
}

@article{bhargavashankar,
  title={Binary quartic forms having bounded invariants, and the boundedness of the average rank of elliptic curves},
  author={Bhargava, Manjul and Shankar, Arul},
  journal={Annals of Mathematics},
  volume={181},
  number ={1},
  pages={191--242},
  year={2015},
  publisher={JSTOR}
}

@article{bhargavagross,
  title={The average size of the 2-Selmer group of Jacobians of hyperelliptic curves having a rational Weierstrass point},
  author={Bhargava, Manjul and Gross, Benedict H.},
  journal={arXiv preprint arXiv:1208.1007},
  year={2012}
}

@article{poonenstoll,
  title={Most odd degree hyperelliptic curves have only one rational point},
  author={Poonen, Bjorn and Stoll, Michael},
  journal={Annals of mathematics},
  volume={180},
  number={3},
  pages={1137--1166},
  year={2014},
  publisher={JSTOR}
}

@article{stoll:independence,
  title={Independence of rational points on twists of a given curve},
  author={Stoll, Michael},
  journal={Compositio Mathematica},
  volume={142},
  number={5},
  pages={1201--1214},
  year={2006},
  publisher={London Mathematical Society}
}

@article{kmrmarkov,
  title={A Markov model for Selmer ranks in families of twists},
  author={Klagsbrun, Zev and Mazur, Barry and Rubin, Karl},
  journal={Compositio Mathematica},
  volume={150},
  number={7},
  pages={1077--1106},
  year={2014},
  publisher={London Mathematical Society}
}

@article{kane,
  title={On the ranks of the 2-Selmer groups of twists of a given elliptic curve},
  author={Kane, Daniel},
  journal={Algebra \& Number Theory},
  volume={7},
  number={5},
  pages={1253--1279},
  year={2013},
  publisher={Mathematical Sciences Publishers}
}

@article{taoziegler,
  title={An inverse theorem for the Gowers $U^{s+1} [N]$-norm},
  author={Green, Ben and Tao, Terence and Ziegler, Tamar},
  journal={Annals of Mathematics},
  pages={1231--1372},
  year={2012},
  publisher={JSTOR}
}

@article{morgansmith,
  title={The Cassels-Tate pairing for finite Galois modules},
  author={Morgan, Adam and Smith, Alexander},
  journal={arXiv preprint arXiv:2103.08530},
  year={2021}
}

@article{kim:triple,
  title={Triple symbols in arithmetic},
  author={Kim, Dohyeong and Morishita, Masanori},
  journal={Research in Number Theory},
  volume={11},
  number={86},
  year={2025},
  publisher={Springer}
}

@article{stevenhagen,
  title={The number of real quadratic fields having units of negative norm},
  author={Stevenhagen, Peter},
  journal={Experimental Mathematics},
  volume={2},
  number={2},
  pages={121--136},
  year={1993},
  publisher={Taylor \& Francis}
}

@article{koymanspagano,
  title={On Stevenhagen's conjecture},
  author={Koymans, Peter and Pagano, Carlo},
  journal={Acta Mathematica},
  year={2025},
  pages={to appear},
  publisher={International Press}
}

@article{koymanspagano:higherredei,
  title={Higher R$\backslash$'edei reciprocity and integral points on conics},
  author={Koymans, Peter and Pagano, Carlo},
  journal={arXiv preprint arXiv:2005.14157},
  year={2020}
}

@article{koymanspagano:highergenus,
  title={Higher genus theory},
  author={Koymans, Peter and Pagano, Carlo},
  journal={International Mathematics Research Notices},
  volume={2022},
  number={4},
  pages={2772--2823},
  year={2022},
  publisher={Oxford University Press}
}

@article{fouvrykluners,
  title={On the 4-rank of class groups of quadratic number fields.},
  author={Fouvry, {\'E}tienne and Kl{\"u}ners, J{\"u}rgen},
  journal={Inventiones mathematicae},
  volume={167},
  number={3},
  pages={455-513},
  year={2007}
}

@article{bostonellenberg,
  title={Random pro-$ p $ groups, braid groups, and random tame Galois groups},
  author={Boston, Nigel and Ellenberg, Jordan S},
  journal={Groups, Geometry, and Dynamics},
  volume={5},
  number={2},
  pages={265--280},
  year={2011}
}

@article{landesmanlevy,
  title={Homological stability for Hurwitz spaces and applications},
  author={Landesman, Aaron and Levy, Ishan},
  journal={arXiv preprint arXiv:2503.03861},
  year={2025}
}

@article{davenportheilbronn,
  title={On the density of discriminants of cubic fields. II},
  author={Davenport, Harold and Heilbronn, Hans Arnold},
  journal={Proceedings of the Royal Society of London. A. Mathematical and Physical Sciences},
  volume={322},
  number={1551},
  pages={405--420},
  year={1971},
  publisher={The Royal Society London}
}

@article{bostonwood,
  title={Non-abelian Cohen--Lenstra heuristics over function fields},
  author={Boston, Nigel and Wood, Melanie Matchett},
  journal={Compositio Mathematica},
  volume={153},
  number={7},
  pages={1372--1390},
  year={2017},
  publisher={London Mathematical Society}
}

@article{liu:cgae,
  title={On the distribution of class groups of abelian extensions},
  author={Liu, Yuan},
  journal={arXiv preprint arXiv:2411.19318},
  year={2024}
}

@article{koymansliu,
  title={Statistics of bad parts of class groups},
  author={Koymans, Peter and Liu, Yuan},
  journal={arXiv preprint arXiv:2512.22849},
  year={2025}
}

@article{wood:liftinginvariant,
  title={An algebraic lifting invariant of Ellenberg, Venkatesh, and Westerland},
  author={Wood, Melanie Matchett},
  journal={Research in the Mathematical Sciences},
  volume={8},
  number={2},
  pages={21},
  year={2021},
  publisher={Springer}
}

@article{park,
  title={On the prime Selmer ranks of cyclic prime twist families of elliptic curves over global function fields},
  author={Park, Sun Woo},
  journal={Compos. Math.},
  pages={to appear},
  year={2022}
}

@article{bbh,
  title={Heuristics for $p$-class towers of imaginary quadratic fields},
  author={Boston, Nigel and Bush, Michael R and Hajir, Farshid},
  journal={Mathematische Annalen},
  volume={368},
  number={1},
  pages={633--669},
  year={2017},
  publisher={Springer}
}

@article{orwbourbaki,
  title={Homology of Hurwitz spaces and the Cohen-Lenstra heuristic for function fields},
  author={Randal-Williams, Oscar},
  journal={S{\'e}minaire Bourbaki},
  volume = {71e},
  number={1162},
  year={2019}
}

@article{sawinwoodroots,
  title={Conjectures for distributions of class groups of extensions of number fields containing roots of unity},
  author={Sawin, Will and Wood, Melanie Matchett},
  journal={arXiv preprint arXiv:2301.00791},
  year={2023}
}

@article{gerthillinois,
  title={The $4 $-class ranks of quadratic extensions of certain imaginary quadratic fields},
  author={Gerth III, Frank},
  journal={Illinois Journal of Mathematics},
  volume={33},
  number={1},
  pages={132--142},
  year={1989},
  publisher={Duke University Press}
}

@article{gerth:4rank,
  title={The 4-class ranks of quadratic fields},
  author={Gerth III, Frank},
  journal={Inventiones mathematicae},
  volume={77},
  number={3},
  pages={489--515},
  year={1984},
  publisher={Springer}
}

@article{gerth:extension,
  title={Extension of conjectures of Cohen and Lenstra},
  author={Gerth III, Frank},
  journal={Exposition. Math},
  volume={5},
  number={2},
  pages={181--184},
  year={1987}
}

@article{gartonthunder,
  title={The distribution of $a$-numbers of hyperelliptic curves in characteristic three},
  author={Garton, Derek and Thunder, Jeffrey Lin and Weir, Colin},
  journal={Finite Fields and Their Applications},
  volume={109},
  pages={102715},
  year={2026},
  publisher={Elsevier}
}

@article{cohenmartinet,
  title={Class groups of number fields: numerical heuristics},
  author={Cohen, Henri and Martinet, Jacques},
  journal={Mathematics of Computation},
  volume={48},
  number={177},
  pages={123--137},
  year={1987}
}

@article{bklpr,
  title={Modeling the distribution of ranks, Selmer groups, and Shafarevich--Tate groups of elliptic curves},
  author={Bhargava, Manjul and Kane, Daniel M and Lenstra, Hendrik W and Poonen, Bjorn and Rains, Eric},
  journal={Cambridge Journal of Mathematics},
  volume={3},
  number={3},
  pages={275--321},
  year={2015},
  publisher={International Press of Boston, Inc. Somerville, MA 02143, USA}
}

@article{bartellenstra,
  title={On class groups of random number fields},
  author={Bartel, Alex and Lenstra Jr, Hendrik W},
  journal={Proceedings of the London Mathematical Society},
  volume={121},
  number={4},
  pages={927--953},
  year={2020},
  publisher={Wiley Online Library}
}

@article{garton,
  title={Random matrices, the Cohen--Lenstra heuristics, and roots of unity},
  author={Garton, Derek},
  journal={Algebra \& Number Theory},
  volume={9},
  number={1},
  pages={149--171},
  year={2015},
  publisher={Mathematical Sciences Publishers}
}

@inproceedings{friedmanwashington,
  title={On the distribution of divisor class groups of curves over a finite field},
  author={Friedman, Eduardo and Washington, Lawrence C},
  booktitle={Th{\'e}orie des nombres: Proceedings of the International Number Theory Conference held at Université Laval, July 5-18, 1987},
  pages={227--239},
  year={1989},
  organization={de Gruyter Berlin}
}

@article{koymansmilovic,
  title={On the 16-Rank of Class Groups of $\Q(\sqrt{-2p})$ for primes $p \equiv 1 \pmod 4$},
  author={Koymans, Peter and Milovic, Djordjo},
  journal={International Mathematics Research Notices},
  volume={2019},
  number={23},
  pages={7406--7427},
  year={2019},
  publisher={Oxford University Press}
}

@article{evw,
  title={Homological stability for Hurwitz spaces and the Cohen-Lenstra conjecture over function fields},
  author={Ellenberg, Jordan S and Venkatesh, Akshay and Westerland, Craig},
  journal={Annals of Mathematics},
  volume={183},
  number={3},
  pages={729--786},
  year={2016},
  publisher={JSTOR}
}

@article{sawinwoodav,
  title={Distributions of unramified extensions of global fields},
  author={Sawin, Will and Wood, Melanie Matchett},
  journal={arXiv preprint arXiv:2602.21032},
  year={2026}
}

@article{heathbrown:congruent1,
  title={The size of Selmer groups for the congruent number problem},
  author={Heath-Brown, D},
  journal={Inventiones Mathematicae},
  volume={111},
  number={1},
  year={1993},
  publisher={Springer-Verlag}
}

@article{heathbrown:congruent2,
  title={The size of Selmer groups for the congruent number problem, II},
  author={Heath-Brown, D},
  journal={Inventiones Mathematicae},
  volume={118},
  number={1},
  year={1994},
  publisher={Springer-Verlag}
}

@article{bjl,
  title={Arakelov class groups of random number fields},
  author={Bartel, Alex and Johnston, Henri and Lenstra Jr, Hendrik W},
  journal={Mathematische Annalen},
  volume={390},
  number={3},
  pages={4405--4428},
  year={2024},
  publisher={Springer}
}

@book{silverman,
  author    = {Silverman, Joseph H.},
  title     = {The Arithmetic of Elliptic Curves},
  series    = {Graduate Texts in Mathematics},
  volume    = {106},
  publisher = {Springer},
  address   = {New York},
  year      = {1986},
}

@article{monsky,
  title={Generalizing the Birch-Stephens theorem: I. Modular curves},
  author={Monsky, Paul},
  journal={Mathematische Zeitschrift},
  volume={221},
  number={1},
  pages={415--420},
  year={1996},
  publisher={Springer}
}

@article{dvir,
  title={On the size of Kakeya sets in finite fields},
  author={Dvir, Zeev},
  journal={Journal of the American Mathematical Society},
  volume={22},
  number={4},
  pages={1093--1097},
  year={2009}
}

@article{golodshafarevich,
  title={On the class field tower},
  author={Golod, Evgeniy Solomonovich and Shafarevich, Igor Rostislavovich},
  journal={Izvestiya Rossiiskoi Akademii Nauk. Seriya Matematicheskaya},
  volume={28},
  number={2},
  pages={261--272},
  year={1964},
  publisher={Russian Academy of Sciences, Steklov Mathematical Institute of Russian~…}
}

\end{document}